\documentclass[11pt]{amsart}

\usepackage{mathrsfs}
\usepackage{amsmath, amscd, amsthm,amssymb, amsfonts, verbatim,subfigure}
\usepackage[mathcal]{eucal}
\usepackage{mathtools}

\usepackage[all]{xy}

\usepackage{mathpazo}
 \DeclareMathOperator{\End}{End}
 
\DeclareMathOperator{\Sym}{Sym} 
\DeclareMathOperator{\Hom}{Hom}
 
\DeclareMathOperator{\Diff}{Diff}

\DeclareMathOperator{\Der}{Der}

\theoremstyle{thm}

\newtheorem{theorem}{Theorem}[section]
\newtheorem{thm-def}{Theorem/Definition}[theorem]
\newtheorem{proposition}[theorem]{Proposition}

\newtheorem{lemma}[theorem]{Lemma}

\newtheorem{corollary}[theorem]{Corollary}

\numberwithin{equation}{section}

\theoremstyle{def}
\newtheorem{definition}[theorem]{Definition}

\theoremstyle{definition}

\newtheorem{remark}[theorem]{Remark}

\newtheorem*{example}{Example}

\theoremstyle{texttheorem}

\usepackage{mathrsfs}
\usepackage{pdfsync}
\usepackage{amscd}
\usepackage[euler-digits,euler-hat-accent]{eulervm}
\usepackage[colorlinks = true]{hyper ref}
\usepackage{color}

\newcommand{\C}{{\mathbb C}}

\newcommand{\K}{{\mathbb K}}

\newcommand{\g}{\mathfrak{g}}
\newcommand{\h}{\mathfrak{h}}

\newcommand{\sfG}{\mathsf{G}}

\newcommand{\gf}{\mathbb{K}}

\newcommand{\sym}{\operatorname{Sym}}

\newcommand{\Ch}{\operatorname{Ch}}
\newcommand{\ch}{\operatorname{ch}}
\newcommand{\Td}{\operatorname{Td}}

\newcommand{\ad}{\operatorname{ad}}

\newcommand{\calD}{\mathcal{D}}
\newcommand {\calA}{\mathcal{A}}

\newcommand{\calU}{\mathcal{U}}
\newcommand{\calV}{\mathcal{V}}

\newcommand{\calO}{\mathcal{O}}
\newcommand{\calL}{\mathcal{L}}
\newcommand{\calJ}{\mathcal{J}}
\newcommand{\calW}{\mathcal{W}}
\newcommand{\calE}{\mathcal{E}}
\newcommand{\calT}{\mathcal{T}}

\title{An index theorem for Lie algebroids}
\author{Arie Blom}
\author{Hessel Posthuma}
\address{\newline
        Hessel Posthuma, {\tt H.B.Posthuma@uva.nl}\newline
         \indent {\rm Korteweg-de Vries Institute for Mathematics,
        University of Amsterdam,
         The Netherlands} 
         \newline
        Arie Blom, {\tt A.Blom@uva.nl}\newline
         \indent {\rm Korteweg-de Vries Institute for Mathematics,
        University of Amsterdam,
         The Netherlands}
}
\date{\today}
\begin{document}
\maketitle
\begin{abstract}
We study Lie algebroids from the point of view noncommutative geometry.
More specifically, using ideas from deformation quantization, we use the PBW-theorem for Lie algebroids to construct 
a Fedosov-type resolution for the associated sheaves of Weyl algebras. This resolution enables us to 
construct a ``character map'' --in the derived category-- from the sheafified cyclic chain complexes to the Chevalley--Eilenberg complex
of the Lie algebroid. The index theorem computes the evaluation of this map on the trivial cycle in terms of the Todd--Chern characteristic class. 
Finally, we show compatibility of the character map with the Hochschild--Kostant--Rosenberg morphism.
\end{abstract}

\section*{Introduction}
Lie algebroids are joint generalizations of tangent sheaves and Lie algebras. These two special cases hint at the most important
features of Lie algebroids: first, considered as a generalized tangent sheaf they exhibit enough structure to do geometry, in particular to set up
a Cartan-type calculus with which one can define a Lie algebroid cohomology theory generalizing de Rham and Lie algebra cohomology.  
Second, Lie algebroids form a notion of (infinitesimal) symmetries whose global version is given by a (Lie) groupoid, similar to the 
integration of Lie algebras to Lie groups. In geometrical situations, e.g. Poisson geometry, the Lie algebroid usually shows up naturally, whereas an integrating groupoid, if it exists,
is not visible at first sight.

In this paper we study Lie algebroids from the point of view of noncommutative geometry via their universal enveloping algebras. The universal enveloping 
algebra of a Lie algebroid is a sheaf of noncommutative algebras generalizing the usual universal enveloping algebra of a Lie algebra as well as the 
sheaf of differential operators associated to the tangent sheaf of a manifold or variety. As a sheaf of noncommutative algebras, the methods of noncommutative
geometry can be applied and in this paper we are interested in the Hochschild and cyclic homology of these objects. Recall that cyclic homology is the noncommutative
version of de Rham cohomology, and our main aim is to compare the cyclic homology of the sheaf of universal enveloping algebras with the Lie algebroid cohomology
of the underlying Lie algebroid mentioned above. The result can phrased as an index theorem.

Let us now describe the main result of this paper. We use a sheaf version of Lie algebroids over locally ringed spaces as in \cite{cvdb} to encompass both the 
differentiable and the holomorphic case. For a Lie algebroid $(\calO_{X},\calL)$ over a ringed space $(X,\calO_{X})$, c.f. definition \S \ref{def-la}, 
we denote by $\calU(\calL)$ its sheaf of universal enveloping algebras. We can twist the whole construction by a locally free $\calO_{X}$-module $\calE$, resulting in a sheaf of algebras 
$\calU(\calL;\calE)$.
Associated to these data are two sheaf complexes: the ``commutative'' Chevalley--Eilenberg complex
$(\Omega^{\bullet}_{\calL},d_{\calL})$ computing the Lie algebroid cohomology and its ``noncommutative'' counterpart, the periodic cyclic complex $(CC_{\bullet}(\calU(\calL))((u)),b+uB)$
 computing the cyclic homology. To phrase the main results of this paper in one theorem, we use the language of derived categories:
\begin{theorem}\label{mainthm}
Let $(\calL,\calO_{X})$ be a Lie algebroid of constant rank $r$ over a ringed space $(X,\calO_{X})$.
\begin{itemize}
\item[$i)$] In the derived category $\calD(X)$ of sheaves on $X$ there exists a canonical morphism
\[
\Phi\in\Hom_{\calD(X)}\left(\left(CC_{\bullet}(\calU(\calL))((u)),b+uB\right),\left(\Omega^{-\bullet}_{\calL}((u))[2r],d_{\calL}\right)\right),
\]
\item[$ii)$] Restricted to the structure sheaf $\calO_{X}\subset\calU(\calL)$, the following diagram commutes:
\[
\xymatrix{CC_\bullet(\calO_{X})\ar[d]^{i}
\ar[rr]^{HKR_{\calL}}&&\Omega^\bullet_\calL\ar[d]^{\cup{\rm Td}_\calL{\rm Ch}_{\calL}}
\\
CC_\bullet(\calU(\calL;\calE))\ar[rr]_{\Phi}&&\Omega^\bullet_\calL}
\]
\end{itemize}
\end{theorem}
In the diagram above, $HKR_{\calL}$ denotes the Hochschild--Kostant--Rosenberg map for the Lie algebroid $(\calL,\calO_{X})$, and ${\rm Td}_\calL$ and ${\rm Ch}_{\calL}$
denote the Todd class of $\calL$, resp. the Chern class of $\calE$. a locally free $\calO_{X}$-module.
In the main body of the paper we have similar statements for Hochschild and all variants of cyclic homology, such as negative cyclic homology.
In the special case of the Lie algebroid given by the holomorphic tangent sheaf to a complex manifold $M$, this result is closely related
to the index theorem \cite{bnt} of Bressler--Nest--Tsygan. In fact, they work on the cotangent bundle $T^{*}M$ and their ``trace density map''
to $\Omega^{\bullet}(T^{*}M)$ is related to ours via the pull-back along the projection $\pi:T^{*}M\to M$. Closer to our approach are the methods 
in \cite{ef} to prove a Riemann--Roch--Hirzebruch theorem for differential operators on complex manifolds. In \cite{ramadoss}, these two approaches 
are combined to prove the index theorem in the form of a commutative diagram as stated above in our main theorem.

The results of this paper appear to be related to those of \cite{cvdb,crvdb12} where a ``Duflo-type'' result for Lie algebroids is proved and 
a Todd class appears as the obstruction to the dual HKR map being multiplicative. However, the precise relationship is at present not known.
To wit, the Hochschild (co)chain complexes, as in \cite{crvdb}, underlying these papers are different than those of this paper: from the point of
view of \cite{kp}, they are more related to {\em Hopf} cyclic homology, instead of the ordinary cyclic homology used in this paper. As a result,
the Todd classes also differ as they live in different cohomologies: in this paper they are defined in the Lie algebroid cohomologies, whereas in 
\cite{cvdb,crvdb12} they are elements of the ordinary sheaf cohomologies of the exterior powers of the Lie algebroid. 
 In case of the 
holomorphic tangent sheaf, a relation between the two theories is given in \cite{ramadoss} in connection with the so-called ``Mukai-pairing''. 
For general Lie algebroids, this pairing is exactly what is still missing from \cite{crvdb12}.

The methods we use to prove the main theorem above are inspired from deformation quantization and the algebraic index theorem \cite{ffs,ppt}.
Crucial in this is the notion of a Fedosov resolution \cite{fedosov}. In our approach we construct such a resolution for the universal enveloping algebra using the 
Poincar\'e--Birkhoff--Witt theorem for Lie algebroids. We take a somewhat different approach than the existing literature \cite{rinehart,nwx,lgsx} to the  
PBW-theorem using the {\em dual} map for the algebra of formal jets on the Lie algebroid. Our approach has the advantage that we get explicit formulae 
for the dependence of the PBW-map on the choice of a connection, and these are crucial in our construction of the character map. 
With the Fedosov resolution, we follow the lines of the proof of the algebraic index theorem in \cite{ffs,ppt,willwacher}, using the fundamental
cyclic cocycle on the Weyl algebra and the index theorem $ii)$ above follows from the local Riemann--Roch theorem for this cocycle.

The structure of this paper is as follows: in Section \ref{sla} we quickly review the theory of Lie algebroids that we need for this paper.
Section \ref{spbw} is devoted to the PBW theorem. As mentioned above, underlying our approach is actually the dual version and in this section we
actually give four versions of this theorem: the dual version for jets, ordinary PBW for the universal enveloping algebra, a ``noncommutative version'' for 
the tensor product of the two, and finally a version with coefficients twisted by a locally free sheaf. In section \ref{sc} we use these PBW-maps to 
construct the Fedosov resolution for the Weyl algebra and construct the character map. Since we use the sheaf version of Lie algebroids, 
we have to use a \v{C}ech-resolution in the construction, and we show that the character map is canonically defined, i.e., independent of choices, in the derived category. 
This settles the first part $i)$ of the main theorem. Finally, in Section \ref{sit} we prove part $ii)$: the index theorem. Because of our use of the \v{C}ech resolution, 
in comparison to the usual algebraic index theorem, we have to evaluate more terms when applying the local Riemann--Roch theorem for the fundamental cocycle on
the Weyl algebra. These extra terms are all related to the derivative of the PBW-map. We end the paper with a discussion of the case of a holomorphic Lie algebroid, 
where, instead of a \v{C}ech resolution, one can use a Dolbeault-type resolution.

\section{Lie algebroids}
\label{sla}
In this section we briefly recall the most important concepts and definitions in the theory of
Lie algebroids, such as their modules and cohomology.

\subsection{Lie--Rinehart algebras}
We start by recalling the notion of a {\em Lie--Rinehart algebra}: this is a pair $(R,L)$ with $R$ a
commutative algebra (over a field $\gf$ of characteristic zero) and an $R$-module $L$ which carries an additional $\gf$-Lie algebra
structure. Furthermore, $L$ acts on $R$ via a map of Lie algebras
$\rho:L\to {\rm Der}_k(R)$ (called the anchor) such that the Leibniz identity holds:
\[
[X_1,rX_2]=r[X_1,X_2]+\rho(X_1)(r)X_2, \quad \mbox{for all}~X_1,X_2\in L,~r\in R.
\]
A morphism $f:(R,L)\rightarrow (S,K)$ of Lie--Rinehart algebras consists of two morphisms $f_1:R\rightarrow S$, $f_2:L\rightarrow K$
respecting the Lie--Rinehart structures, e.g. $\rho_2(f_2(X))f_1(r)=f_1(\rho_1(X)(r))$.

Given an $R$-module $M$, an $L$-connection is given by a $\mathbb{K}$-linear map
$\nabla:M
\to
{\rm Hom}_R(L,M)$,
written as $\nabla_l:M\to M$, that satisfies the {\em Leibniz identity}
\[
\nabla_X(rm)=r\nabla_X(m)+\rho(X)(r)m,\quad r\in R,~X\in L,~m\in M.
\]
Many of the constructions in this paper depend on the possible choice of a connection.
The following Lemma is probably well-know, we decided to give a short proof for completeness:
\begin{lemma}\label{proj-conn}
Let $(L,R)$ be a Lie--Rinehart algebra with $R$ unital. Then a projective $R$-module admits an $L$-connection.
\end{lemma}
\begin{proof}
Because $R$ is unital, the sequence of $R$-modules
\[
R\otimes_\gf M\longrightarrow M\longrightarrow 0,
\]
where the $R$ action on $R\otimes_\gf M$ is multiplication in $R$ and the first arrow denotes the $R$-module action on $M$,
is exact. Applying ${\rm Hom}_R(M,-)$, the sequence
\[
{\rm Hom}_R(M,R\otimes_\gf M)\longrightarrow {\rm Hom}_R(M,M)\longrightarrow 0
\]
is exact by the fact that $M$ is projective over $R$. Therefore there exists an element
$s\in {\rm Hom}_R(M,R\otimes_\gf M)$ which projects to ${\rm id}_M\in {\rm Hom}_R(M,M)$.
For $m\in M$, we write $s(m)=s(m)_{(1)}\otimes s(m)_{(2)}\in R\otimes_\gf M$ in obvious Sweedler-like
notation. With this we can define
\[
\nabla_X(m):=X(s(m)_{(1)})s(m)_{(2)},\quad X\in L, m\in M.
\]
Its is easy to check that the Leibniz identity holds true:
\begin{align*}
\nabla_X(rm)&=X(rs(m)_{(1)})s(m)_{(2)}\\
&=X(r)s(m)_{(1)}s(m)_{(2)}+rX(s(m)_{(1)})s(m)_{(2)}\\
&=X(r)m+\nabla_X(m).
\end{align*}
So $\nabla$ as defined above indeed defines an $L$-connection.
\end{proof}
The curvature $R(\nabla)\in \End_{R}(M)\otimes_R\bigwedge^2_RL^\vee$ of a connection $\nabla$ on $M$ is given by 
\begin{equation}
\label{curv-conn}
R(\nabla)(X,Y):=[\nabla_{X},\nabla_{Y}]-\nabla_{[X,Y]},\quad \mbox{for all}~X,Y \in L.
\end{equation}
When the connection is flat, i.e., $R(\nabla)=0$, 
we call the pair $(M,\nabla)$ a {\em module} over the Lie--Rinehart algebra $(L,R)$. For such a
module we can define the associated de Rham complex by
\[
\Omega^\bullet_L(M):=M\otimes_{R}\bigwedge^\bullet_R L^*,
\]
with differential $\nabla:\Omega^p_L(M)\to \Omega^{p+1}_L(M)$ defined by the Koszul formula
\begin{align}
\label{Koszul}
\nabla \omega(X_0,\ldots,X_p)&= \sum_{i=0}^p(-1)^i\nabla_{X_i}\big(\omega(X_0,\ldots,\hat{X}_i,\ldots, X_p)\big)\\
&\quad+\sum_{i<j}(-1)^{i+j}\omega([X_i,X_j],l_0,\ldots,\hat{X}_i,\ldots,\hat{X}_j,\ldots,X_p),\nonumber
\end{align}
where the $~\hat{}~$ means omission of the argument from a linear map. One checks that
$\nabla^2=0$, the resulting complex is called the Chevalley--Eilenberg (or de Rham) complex of the $(L,R)$-module $M$. 
Its cohomology, denoted as $H^\bullet(L;M)$, is called the {\em Lie--Rinehart} cohomology of the $(L,R)$-module $M$.

\subsection{Lie algebroids}
In this paper we are interested in Lie--Rinehart algebras arising in geometry, which are given by
sheaves of Lie--Rinehart algebras over ringed spaces $(X,\calO_X)$.
\begin{definition}\label{def-la}
A Lie algebroid over $(X,\calO_X)$ is a sheaf of modules such that $(\calL(U),\calO_X(U))$
is a Lie--Rinehart pair for each open $U\subset X$. 
Moreover, the restriction morphisms are required to be morphisms of Lie--Rinehart pairs.
\end{definition}
In the following we assume that $\calL$ is locally free and of constant rank $r$.

For a Lie algebroid $(\calL,\calO_{X})$ as defined above, the global sections $(\calL(X),\calO_{X}(X))$ form a 
Lie--Rinehart algebra.
Conversely, given a Lie Rinehart algebra $(L,R)$ with $L$ projective over $R$, the induced sheaf $\calL$ over $X={\rm Spec}(R)$
is locally free and defines a Lie algebroid which admits a connection.
For a general Lie algebroid, $\calL$ need not be projective in the category of sheaves of $\calO_X$-modules, need not to admit a connection.

Let $\calE$ be a locally free $\calO_{X}$-module. The notion of a connection for a Lie-Rinehart algebra readily generalizes to
that of an $\calL$-connection $\nabla^{\calE}$ on $\calE$ by requiring the underlying map to be a morphism of sheaves. When such a connection
is flat, we say that $(\calE,\nabla^{\calE})$ is an $\calL$-module. In this case the connection extends to a differential of degree $1$ on 
$\Omega^{\bullet}_{\calL}(\calE):=\bigwedge^{\bullet}_{\calO_{X}}\calL^{\vee}\otimes_{\calO_{X}}\calE$. The hypercohomology of this complex of sheaves
defines the Lie algebroid cohomology of $\calL$ with values in $\calE$; 
\[
H^{\bullet}(\calL;\calE):=\mathbb{H}(\Omega^{\bullet}_{\calL}(\calE),\nabla^{\calE}).
\]
In this paper we shall mostly use \v{C}ech-resolutions to compute such cohomology groups.
\begin{example}
Let us briefly mention some examples.
\begin{itemize}
\item[$i)$] A Lie algebra is a Lie algebroid over a point.
\item[$ii)$] The tangent sheaf over a smooth manifold $X$ is a Lie algebroid with anchor equal to the identity. When $X$ is a complex manifold, 
we can consider the holomorphic tangent bundle $\calT_X$ as a Lie algebroid over the structure sheaf $\calO_X$ of germs of holomorphic functions.
For these Lie algebroids, the universal enveloping algebra (see below) is given by the sheaf of (holomorphic) differential operators.
Sometimes, we can even restrict the structure sheaf even more: for example, on $\C^n$, the Lie algebroid of polynomials and polynomial vector
fields has the Weyl algebra (c.f.\ Appendix \ref{weyl}) as its universal enveloping algebra
\item[$iii)$] A foliation on a manifold determines a Lie algebroid by means of the sheaf of vector fields tangent to the foliation. Again we can also 
consider this in the holomorphic category. Its universal enveloping algebra is given by the differential operators along the leaves.
\item[$iv)$] Let $D\subset X$ be a hypersurface in a complex manifold. The sheaf $\calT_X(-\log D)$ of vector fields tangent to $D$
defines a Lie algebroid. 
\item[$v)$] Let $E\to X$ be a holomorphic vector bundle over a complex manifold. The {\em Atiyah algebroid} consists of the subsheaf of
$\Diff_X^{\leq 1}(E)$, the differential operators of order $\leq 1$ on $E$, with scalar symbol.  When $E$ is a line bundle, all differential operators have 
scalar symbol, and the universal enveloping algebra of the Atiyah algebroid are the so-called twisted differential operators.
\end{itemize}
\end{example}

\subsection{The universal enveloping algebra}
In this subsection we will work, for simplicity, with a Lie--Rinehart algebra $(R,L)$. All the constructions can easily be generalized in the 
sheaf setting of a Lie algebroid. To a pair $(R,L)$ one can associate its universal enveloping algebra $\calU(L,R)$,
constructed as follows: consider the direct sum $R\oplus L$ and equip this with the Lie bracket
\[
[(r_1,X_1),(r_2,X_2)]:=(\rho(X_1)(r_2)-\rho(X_2)(r_1),[X_1,X_2]).
\]
The universal enveloping algebra of this Lie algebra is denoted by $\calU(L\oplus R)$. By definition
it comes equipped with a canonical injection $i:R\oplus L\to \calU(L\oplus R)$ and we write
$\calU(L\oplus R)^+$ for the subalgebra generated by $i(R\oplus L)$. The universal
 enveloping algebra $\calU(L,R)$ of the Lie--Rinehart algebra $(R,L)$ is defined as
\[
\calU(L,R):=\calU(L\oplus R)^+\slash I
\]
where $I$ is the two-sided ideal generated by $i(r_1,0)i(r_2,X)
-i(r_1r_2,r_1X)$
with $r_1,r_2\in R$, $X\in L$. This defines a unital algebra equipped with a $\gf$-algebra morphism $i_R:R\to \calU(L,R)$, as well as a morphism
$i_L:L\to {\rm Lie}(\calU(L,R))$ of $\gf$-Lie algebras, subject to the conditions
\[
i_R(r) i_L(X) = i_L(rX), \qquad i_L(X) i_R(r) - i_R(r) i_L(X) =
i_R(X(r)), \quad r \in R, \ X \in L.
\]
It has the following universal property:
for any triple $(A, \phi_L, \phi_R)$ consisting of a $\gf$-algebra
$A$, a  homomorphisms $\phi_R:R \to A$, and a Lie algebra morphism $\phi_L:L\to{\rm Lie}(A)$, such that for all
$r \in R, \ l \in L$
\[
\phi_R(r) \phi_L(X) = \phi_L(rX), \qquad \phi_L(X) \phi_R(r) -
\phi_R(r) \phi_L(X) = \phi_R(X(r)),
\]
there is a unique morphism $\Phi: \calU(L,R)  \to A$
of $k$-algebras such that $\Phi\circ i_R = \phi_R$ and $\Phi\circ i_L = \phi_L$.
This property shows that the
natural functor
$
\mathsf{Mod}(\calU(L,R))\to\mathsf{Mod}(R,L)
$
is an equivalence of categories. The maps $\Delta: \calU(L,R)\rightarrow \calU(L,R)\otimes_R \calU(L,R)$ and 
$\epsilon: \calU(L,R)\rightarrow R$ defined by
\begin{align*}
\Delta(r)=& r\otimes 1,\quad\quad\quad\quad\quad r\in R\\
\Delta(X)=& X\otimes 1+1\otimes X, \quad\; X\in L\\
\epsilon(D)=& D(1), \quad\quad\quad\quad\quad D\in \calU(L,R)
\end{align*} 
give $\calU(L,R)$ the structure of a cocommutative $R$-coalgebra. For a comprehensive treatment,
consult \cite{kp,crvdb}. We will write $\Delta(D)=D_{(1)}\otimes D_{(2)}$  using Sweedler notation.

\subsection{The Jet algebra}\label{jet}
The Jet algebra is simply defined as the dual of the universal enveloping algebra:
\[
\calJ(L,R):={\rm Hom}_R(\calU(L,R),R),
\]
where $\calU(L,R)$ is considered as an $R$-module via left multiplication.
One can easily dualize the structures on $\calU(L,R)$.
The product on $\calJ(L,R)$ is defined by
\begin{equation}
\label{prod-jets}
(\phi_1\phi_2)(D):=\phi_1(D_{(1)})\phi_2(D_{(2)}),\quad \phi_1,\phi_2\in\calJ(L,R),~D\in\calU(L,R).
\end{equation}
By cocommutativity of $\calU(L,R)$,
this defines a commutative algebra structure on $\calJ(L,R)$ for which the counit $\epsilon$ of
$\calU(L,R)$ is a unit. The left and right $R$-module structures on $\calU(L,R)$ endow $\calJ(L,R)$ with two morphisms $\alpha_i:R\to\calJ(L,R),~i=1,2$ defined as
\[
\begin{split}
\alpha_1(r)(D)&:=\epsilon(rD)=r\epsilon(D),\\
\alpha_2(r)(D)&:=\epsilon(Dr)=D(r), \qquad r \in R, \ D \in \calU(L,R).
\end{split}
\]
One easily checks that we have $(\alpha_1(r)\phi)(D)=r\phi(D)$ and $(\alpha_2(r)\phi)(D)=\phi(Dr)$ for
$\phi\in \calJ(L,R)$, $r\in R$ and $D\in \calU(L,R)$.
Remarkably, there are two flat $L$ connections on $\calJ(L,R)$:
\begin{subequations}
\begin{align}
\label{module1}
\nabla^{(1)}_l\phi(D)&:=\phi(lD)-l(\phi(D))\\
\label{module2}
\nabla^{(2)}_l\phi(D)&:=\phi(Dl),\qquad l\in L,~\phi\in\calJ(L,R),~D\in\calU(L,R).
\end{align}
\end{subequations}
They define, combined with $\alpha_1$ and $\alpha_2$,  two commuting $\calU(L,R)$-module structures on $\calJ(L,R)$,
written $\cdot_{1}$ and $\cdot_{2}$. The second is given by $D\cdot_2\phi(E)=\phi(ED)$.
In the sheaf setting of Lie algebroids $(\calL,\calO_X)$ as in Definition \ref{def-la}, we shall write the two $\calO_X$-module structures
on the sheaf of jets $\calJ(\calL)$ as $\calO_{X_1}$ and $\calO_{X_2}$.

\section{The Poincar\'e--Birkhoff--Witt theorem}
\label{spbw}
Let $(L,R)$ be a Lie--Rinehart algebra with $L$ projective over $R$. The universal enveloping algebra $\calU(L,R)$ carries a canonical increasing filtration
\begin{equation}
\label{filt-un}
0\subset F_0\calU(L,R)\subset F_1\calU(L,R)\subset F_2\calU(L,R)\subset\ldots,
\end{equation}
where $F_0\calU(L,R):=R$ and $F_p\calU(L,R)$ is generated by $i(L)^p$. 
The Poincar\'e--Birkhoff--Witt theorem for Lie--Rinehart algebras, proved in \cite{rinehart}, asserts that the
symmetrization map
\begin{equation}
\label{sym}
X_1\otimes\ldots\otimes X_p\mapsto \frac{1}{p!}\sum_{\sigma\in S_p}X_{\sigma(1)}\cdots X_{\sigma(p)},\quad X_{i}\in L,~i=1,\ldots,p,
\end{equation}
defines an isomorphism of graded algebras
\[
\sym_RL\stackrel{\cong}{\longrightarrow}{\rm Gr}(\calU(L,R)).
\]
Here $\sym_RL$ denotes the symmetric algebra of $L$ over $R$ and it is important to realize
that the symmetrization map is not well-defined with codomain
$\calU(L,R)$ because $i(R)$ and $i(L)$ do not commute in $\calU(L,R)$. 
(For Lie algebras, with $R=\mathbb{K}$ this problem does not occur.)
For Lie algebroids in the $C^{\infty}$ category, this problem was circumvented
in \cite{nwx} by using a (local) groupoid integrating the Lie algebroid and with this the authors
were able to define a map to the universal enveloping algebra inducing the symmetrization map above on
the graded quotient. In this section we shall prove a Poincar\'e--Birkhoff--Witt theorem for Lie algebroids
in similar spirit with the differences that: 1) instead of using an integrating (local) groupoid we use
the {\em formal} integration given by the jet algebra, and 2) we actually construct the --in our opinion-- more fundamental {\em dual}
PBW-map with domain the jet algebra.

\subsection{The PBW theorem for the jet algebra}\label{pbw}
In this section we construct an explicit PBW-map for the jet algebra of a Lie algebroid, depending on the choice of connection. 
Let $\calL$ be a locally free Lie algebroid of rank $r$ over $(X,\calO_X)$. Then $\calU(\calL)$ and $\calJ(\calL)$ are the sheaves
of algebras given by assigning the universal enveloping
algebra and the jet algebra of the respective Lie--Rinehart algebra to each open subset of $X$. 
We assume that the sheaf $\calL$ admits an $\calL$-connection $\nabla^{\calL}$.

Consider the sheaf of symmetric algebras
\[
\widehat{\sym}_{\calO_{X}}\calL^{\vee},
\]
where the notation means that we have formally completed with respect to the degree. 

It is the algebra of jets associated to the Lie algebroid $\calL$ if we put the Lie bracket $[~,~]$ and the anchor map equal to zero.
The $\calL$-connection $\nabla^\calL$ on $\calL$ induces an $\calL$-connection on the sheaf of symmetric algebras, also written as $\nabla^{\calL}$,
and we use this to define its symmetrized version $\nabla^\calL_s$ defined by 
\begin{equation}
\label{symmetrized}
\nabla^{\calL}_s\alpha(X_1\ldots X_n):=\sum_{i=1}^n(\nabla^{\calL}_{X_i}\alpha)(X_1\ldots \widehat{X}_i\ldots X_n).
\end{equation}
This defines a degree $1$ derivation of $\widehat{\sym}(\calL)^{\vee}$. Consider now 
\[
\calJ(\calL)\otimes_{\calO_{X_2}}\widehat{\sym}_{\calO_X}(\calL^{\vee})
\]
equipped with the connection $\nabla^{(2)}\otimes 1+1\otimes \nabla^\calL$. 
The symmetrized version of this connection, defined
by the similar formula to \eqref{symmetrized} above, we write as $D$. Again, it is a derivation of degree $1$ with respect to the 
degree of symmetric tensors in $\widehat{\sym}_{\calO_X}(\calL^{\vee})$.
The PBW-map is now defined as the composition
\[
\calJ(\calL)\stackrel{-\otimes 1}{\longrightarrow}\calJ(\calL)\otimes_{\calO_{X_2}}\widehat{\sym}_{\calO_X}(\calL^{\vee})\stackrel{e^{D}}{\longrightarrow}\calJ(\calL)\otimes_{\calO_{X_2}}\widehat{\sym}_{\calO_X}(\calL^{\vee})
\stackrel{{\rm ev}_{1}}{\longrightarrow} \widehat{\sym}_{\calO_X}(\calL^{\vee}),
\]
where ${\rm ev}_{1}$ means evaluation of a jet at $1\in\calU(\calL)$. Written out this defines a map $j^{\rm PBW}_\nabla :\calJ(\calL)\rightarrow \widehat{\sym}_{\calO_X}(\calL^{\vee})$ given by the formula
\begin{equation}
j^{\rm PBW}_\nabla(\phi):=\sum_{n=0}^\infty\frac{D^{n}\phi}{n!}(1)\label{jnabla}
\end{equation}
\begin{theorem}
\label{dual-pbw}
Let $\calL$ be a locally free sheaf of Lie--Rinehart algebras of constant rank which admits an $\calL$ connection $\nabla^{\calL}$, then
the map
\[
j^{\rm PBW}_\nabla :\calJ(\calL)\rightarrow \widehat{\sym}_{\calO_X}(\calL^{\vee}),
\]
defined in \eqref{jnabla} is an $\calO_{X_{1}}$ linear isomorphism of sheaves of algebras.
Conversely, any such isomorphism defines an $\calL$-connection on $\calL$.
\end{theorem}
\begin{proof}
The $\calO_{X_1}$-linearity follows from fact that the actions of $\nabla^{(2)}$ and $\alpha_1$ commute on $\calJ(\calL)$. 
To see that it is a morphism of algebras one observes that $-\otimes 1$, $ev_1$ and the exponent of a degree 1 derivation
on an algebra all have this property. Now we only have to check that the map is an isomorphism, which we do locally. 
Both $\calJ(\calL)$ and $\widehat{\sym}(\calL^\vee)$ carry a natural descending filtration dual to the ascending ones
on $\calU(\calL)$ respectively $\sym(\calL)$, see \eqref{filt-un}. Since $\calJ(\calL)$ and $\widehat{\sym}(\calL^\vee)$  
are complete with respect to those filtrations, it is enough to show that the map is an isomorphism 
on the associated graded spaces. Let $\phi\in I^k$, i.e. $\phi(D)=0$ for $D\in F_j\calU(L,R)$ where $j<k$. Observing that 
$\nabla^{(2)}_X: I^k\rightarrow I^{k-1}$, it follows quite easily that 
\begin{equation}\label{graded-pbw}
j_\nabla^{\rm PBW}(\phi)(X_1\ldots X_k)=\frac{1}{k!}\sum_{\sigma\in S_k}\phi(X_{\sigma(1)}\ldots X_{\sigma(k)}),
\end{equation}
which proves the theorem.
\end{proof}

\subsection{The PBW-theorem}
The usual PBW-theorem for the universal enveloping algebra $\calU(\calL)$ of a Lie algebroid is obtained by duality:
\begin{corollary}
\label{upbw}
Let $\calL$ be as in the previous theorem, then
the map $\rho^*_\nabla :\sym(\calL)\rightarrow \calU(\calL)$ defined by
\[
\phi(j^*_\nabla(\alpha))=\left<\alpha,j_\nabla(\phi)\right>,\quad \alpha\in \sym(\calL),~\phi\in\calJ(\calL),
\]
is an isomorphism of filtered comonoids in the category of $\calO_X$-modules. Furthermore, the induced map
on the graded quotient is the symmetrization map \eqref{sym}.
\end{corollary}
\begin{proof}
By duality, c.f. \cite[Lemma 5.1]{crvdb}, we have that
\[
\calU(\calL)\cong{\rm Hom}^{\rm cont}_{\calO_{X_1}}(\calJ(\calL),\calO_X),
\]
where the right hand side means taking the $\calO_X$-linear homomorphisms which are continuous
with respect to the $I$-adic topology. Since this filtration is dual to the natural
filtration on $\calU(\calL)$ and since $j_\nabla$ respects the filtrations
the statement follows. The map $j_{\nabla}^*$ respects the comonoid structures because $j_\nabla$ respects the product structures
and the last claim follows from formula \eqref{graded-pbw}.
\end{proof}
\begin{remark}
Our approach to the PBW-theorem via the dual map defined on jets can be related to that of \cite{lgsx} which discusses a more general version 
for inclusions of Lie algebroids. For this, consider $X:=X_{0}\ldots X_{k}\in\sym^{k}(\calL)$ and $\phi\in J(\calL)$, and write out the PBW-formula of Corollary \ref{upbw} above:
\begin{align*}
\phi(j^{PBW}_{\nabla}(X))&=\left<j^{PBW}_{\nabla}(\phi),X\right>\\
&=\frac{1}{(k+1)!}\left<D^{k}\phi(1),X\right>\\
&=\frac{1}{(k+1)!}\sum_{i=0}^{k}\left<(\nabla^{(2)}_{X_{i}}+\nabla^{\calL}_{X_{i}})(D^{n}\phi),X_{0}\ldots\hat{X}_{i}\ldots X_{k}\right>(1),
\end{align*}
where the hat means omission from the argument.
Therefore we find the recursive formula
\[
j^{PBW}_{\nabla}(X_{0}\ldots X_{k})=\frac{1}{k+1}\sum_{i=0}^{k}\left(X_{i}\cdot j^{PBW}_{\nabla}(X_{0}\ldots\hat{X}_{i}\ldots X_{k})
-j^{PBW}_{\nabla}(\nabla_{X_{i}}(X_{0}\ldots\hat{X}_{i}\ldots X_{k}))\right).
\]
This is precisely the definition of the PBW-morphism in \cite[Remark 2.3]{lgsx} for the inclusion of the trivial Lie algebroid into $\calL$. 
\end{remark}
\begin{remark}
The construction of the PBW isomorphism depends on the choice of a connection. In general there is of course no 
canonical choice of a such a connection, but an exception is given by the special case of a Lie algebra $\g$: here we
have the adjoint representation $\ad: \g\rightarrow \End(\g)$. It is easy to check that in this case, using induction, the term 
\[
-\frac{1}{k+1}\sum_{i=0}^{k}j^{PBW}_{\ad}(\ad(X_{i})(X_{0}\cdots\hat{X}_{i}\cdots X_{k}))
\]
vanishes for each $k$, hence the PBW isomorphism $\sym(\g)\rightarrow \calU(\g)$ induced by the adjoint representation is the symmetrization map.
\end{remark}

\subsection{The noncommutative PBW-theorem}
In this section we shall describe a noncommutative version of the PBW-theorem for Lie algebroids.
We assume that $\calL$ is a Lie algebroid over a locally ringed space $(X,\calO_X)$. The tensor product
$\calJ(\calL)\otimes_{\calO_{X_2}}\calU(\calL)$ has, using the coproduct on $\calU(\calL)$ together
with the second $\calU(\calL)$-module structure \eqref{module2} on $\calJ(\calL)$, an associative product structure 
given by
\begin{equation}
\label{mult-fd}
(\phi\otimes D)\cdot(\psi\otimes E):=\phi D_{(1)}(\psi)\otimes D_{(2)}E.
\end{equation}
In \cite[Lemma 4.3.2]{cvdb} it is proved that $\calJ(\calL)\otimes_{\calO_{X_2}}\calU(\calL)\cong\Diff_{O_{X_1}}(\calJ(\calL))$ via the following morphism of sheaves of algebras:
\begin{equation}\label{formal-diff}
\begin{split}
 (\phi\otimes D)(\psi):=\phi D(\psi),\quad
\phi,\psi\in\calJ(\calL),~D\in\calU(\calL).
\end{split}
\end{equation}
The noncommutative PBW-theorem is now given by:
\begin{theorem}
\label{ncpbw}
Let $\calL$ be a Lie algebroid over a  locally ringed space $(X,\calO_X)$. Any $\calL$-connection on $\calL$
induces an isomorphism of sheaves of algebras
\[
\tilde{j}_\nabla:\calJ(\calL)\otimes_{\calO_{X_2}}\calU(\calL)\stackrel{\cong}{\longrightarrow} \Diff_{\calO_X}(\widehat{\sym}(\calL^{\vee}))
\]
\end{theorem}
\begin{proof}
The PBW-Theorem \ref{dual-pbw} gives $\calJ(\calL)\cong \widehat{\sym}
(\calL^\vee)$, and since this map is $\calO_{X_1}$-linear it follows that
\[
j_\nabla: \Diff_{\calO_{X_{1}}}(\calJ(\calL))\stackrel{\cong}{\longrightarrow}\Diff_{\calO_X}(\widehat{\sym}(\calL^\vee).
\]
Explicitly, the map is given by $j_\nabla(K)(\alpha)=j_\nabla(K(j_\nabla^{-1}(\alpha)))$ for $K\in \Diff_{\calO_{X_{1}}}(\calJ(\calL))$
and $\alpha\in \widehat{\sym}(\calL^\vee)$. 
\end{proof}
The sheaf $ \Diff(\widehat{\sym}(\calL^{\vee}))$ of fiberwise differential operators on the sheaf of commutative algebras $\widehat{\sym}(\calL^{\vee})$
is nothing but a sheaf of (formally completed) Weyl algebras. Indeed, the Weyl algebra as described in Appendix \ref{weyl}, is isomorphic to the algebra of polynomials
equipped with the Moyal--Weyl product \eqref{moyal-weyl}. Likewise, we have a canonical isomorphism
\[
\Diff_{\calO_{X}}(\widehat{\sym}(\calL^{\vee}))\cong \widehat{\sym}(\calL^\vee)\otimes_{\calO_{X}}\sym(\calL)=:\calW(\calL),
\]
and remark that on this (partially) completed Weyl algebra, the Moyal--Weyl product \eqref{moyal-weyl} still makes sense, 
since the expansion of the exponential in the formula \eqref{moyal-weyl} is still finite.  
\subsection{The derivative of the PBW-map}
The main advantage of our approach to the PBW-theorem is that with the explicit formula \eqref{jnabla} for the dual PBW-map on the algebra of jets,
we can easily compute its derivative to see in what way it depends on the choice of a connection. Let $\nabla_1$ and $\nabla_2$ be two $\calL$-connections on
$\calL$. We form the family of connections relating the these two: $\nabla_t:=t\nabla_2+(1-t)\nabla_1$ where $t\in [0,1]$. Clearly, this construction induces a family $j^{PBW}_t:=j^{ PBW}_{\nabla_t}$ of $PBW$-isomorphisms.
\begin{proposition}\label{depconn}
Let $\nabla_t=t\nabla_2+(1-t)\nabla_1$. For each $t\in[0,1]$ the element 
\[
\theta_t:=\frac{1-\exp\nabla_t}{\nabla_t}\left(\frac{d\nabla_t}{dt}\right)
\] 
is a derivation of $\widehat{\Sym}(\calL^\vee)$, and we have that 
\[
j^{PBW}_t=\exp(\theta_t)\circ j^{PBW}_0.
\]
\end{proposition}

\begin{proof}
Let  $\gamma:=\nabla_2-\nabla_1$ be the difference of the two connections, it acts as a derivation on $\widehat{\Sym}(\calL^\vee)$.
By induction, one shows that $\nabla_t^n\gamma=[\nabla_t,\nabla_t^{n-1}\gamma]$ is a derivation since the commutator of two
derivations is a derivation. It follows that $\theta\in \text{Der}(\widehat{\Sym}(\calL^\vee))$.\\
Recall from the definition that $j^{PBW}_t={\rm ev}_1\circ \exp(D_t)\circ -\otimes 1$. Note that
\[
D_{t}:=\nabla^{(2)}\otimes 1+1\otimes\nabla_t
\]
is a derivation of degree $1$ on $\calJ(\calL)\otimes_{\calO_{X_2}} \widehat{\Sym}(\calL^\vee)$ with respect to the degree in $\widehat{\Sym}(\calL^\vee)$. 
Then the exponential $\exp D_{t}$ is a group like differential operator on the commutative algebra $\calJ(\calL)\otimes_{\calO_{X_2}} \widehat{\Sym}(\calL^\vee)$. 
The derivative of this map is given by the following well-know formula, which follows from a straightforward computation:
\begin{align*}
\frac{d}{dt} (\exp{D_t})=\frac{\exp{\rm ad}_{D_{t}}-1}{\rm ad_{D_{t}}}\left(\frac{d D_{t}}{dt}\right)\circ \exp D_{t}
\end{align*}
Note that $D_t$ and $d D_t\slash dt$ do not commute in general.
Next, using $d D_t\slash dt=\gamma$ and $[D_{t},\gamma]=[\nabla_t,\gamma]=\nabla_{t}\gamma$, we find that
\[
\frac{d}{dt}j_t^{PBW}=\theta\circ j_t^{PBW}.
\]
Because $\ad(\gamma)(\gamma)=0$ when symmetrized,
$[\nabla_{t},\gamma]=\nabla_0\gamma$,  and therefore $\theta_t$ is independent of $t$. 
This implies $j^{PBW}_t=\exp(t\theta)\circ j^{PBW}_{0}$ because the left and right hand side satisfy the same ODE in $t$.
\end{proof}

\subsection{Twisting by a locally free sheaf}
In this final section on the PBW-theorem, we discuss how to incorporate a locally free sheaf $\calE$ of $\calO_{X}$-modules of constant rank into the theory.
Following \cite{cvdb}, we define the space of order $k$ $\calL$-jets of $\calE$ to be 
\[
\calJ^k(\calL;\calE):=\calJ^{k}_{\calL}\otimes_{\calO_{X_2}}\calE.
\]
The first oder jets fit into a short exact sequence
\[
0\to\calL^\vee\otimes\calE\to \calJ^1(\calL;\calE)\to\calE\to 0,
\]
and a  splitting $\sigma:\calE\to\calJ^1(\calL;\calE)$ of this sequence in the category of sheaves of 
$\calO_{X_{1}}$-modules is equivalent to an $\calL$-connection on $\calE$.
 As for the ordinary jets, we 
denote the full space of infinite jets of $\calE$ by $\calJ(\calL;\calE)$. Twisting by $\calE$ means that $\calJ(\calL;\calE)$ is no longer an algebra.
Formula \eqref{prod-jets} only defines a $\calJ(\calL)$-module structure. 
We define the universal enveloping algebra as
\[
\calU(\calL;\calE):={\rm Hom}^{\rm cont}_{\calO_{X_1}}\left(\calJ(\calL;\calE),\calE\right).
\]
\begin{lemma}
There is a canonical isomorphism of $\calU(\calL)$-modules
\[
\calU(\calL;\calE)\cong\calE\otimes_l\calU(\calL)\otimes_r\calE^{\vee}.
\]
\end{lemma}
\begin{proof}
First we note that the $l$ and $r$ indicate the left and right $\calO_X$-module structures on $\calU(\calL)$ that are used to define the tensor products.
Then, for $e_1,e_2\in \calE$, $f\in \calE^\vee$, $D\in \calU(\calL)$ and $\phi\in \calJ(\calL)$, the map
\begin{align*}
 \calE\otimes_\gf\calU(\calL)\otimes_\gf\calE^{\vee}&\longrightarrow {\rm Hom}^{\rm cont}_{\calO_{X_1}}\left(\calJ(\calL;\calE),\calE\right),\\
 e_1\otimes D\otimes f&\longmapsto (\phi\otimes e_2\mapsto \alpha_2(f(e_2))\phi(D)e_1)
\end{align*}
is well defined, and moreover has kernel $(re\otimes D\otimes f-e\otimes rD\otimes f,e\otimes Dr\otimes f-e\otimes D\otimes rf)$ which is
the defining ideal for the tensor products, and this establishes the claim.
\end{proof}
Following our approach in the untwisted case, the PBW-isomorphism should first be proved for the sheaf of jets:

\begin{theorem}
Let $\calL$ be a Lie algebroid over a ringed space $(X,\calO_{X})$, and $\calE$ a locally free $\calO_{X}$-module of
constant finite rank. Any pair $(\nabla^{\calL},\nabla^{\calE})$ of $\calL$-connections on $\calL$ and $\calE$ induces
an $\calO_{X_1}$-linear isomorphism of $\calJ(\calL)$-modules
\[
j_{\nabla^{\calL},\nabla^{\calE}}:\calJ(\calL;\calE)\to \widehat{\sym}_{\calO_X}(\calL^\vee)\otimes_{\calO_X}\calE
\]
\end{theorem}
\begin{proof}
First we note that by an isomorphism of $\calJ(\calL)$-modules, we mean that
\begin{equation}\label{JL-morph}
j_{\nabla^\calL,\nabla^\calE}(\phi(\psi\otimes e))=j_{\nabla^\calL}(\phi)j_{\nabla^\calL,\nabla^\calE}(\psi\otimes e),\quad\quad \phi\in \calJ(\calL),\quad \psi\otimes e\in \calJ(\calL;\calE).
\end{equation}
The connections $\nabla^\calL, \nabla^\calE,\nabla^{(2)}$ define an connection on 
$\widehat{\sym}(\calL^\vee)\otimes_{\calO_{X_2}}\calJ(\calL;\calE)$
and after symmetrizing moreover a map 
\[
D: \widehat{\sym}(\calL^\vee)\otimes_{\calO_{X_2}}\calJ(\calL;\calE)\rightarrow \widehat{\sym}(\calL^\vee)\otimes_{\calO_{X_2}}\calJ(\calL;\calE)
\]
which has the following Leibniz property:
\[
D(rs)=D(r)s+rD(s),\quad\quad r \in \widehat{\sym}(\calL^\vee)\otimes_{\calO_{X_2}}\calJ(\calL),\quad s\in \widehat{\sym}(\calL^\vee)\otimes_{\calO_{X_2}}\calJ(\calL;\calE).
\]
Here $D$ stands for the operator from \ref{pbw} and the operator
just defined, and the multiplication is the natural one. The PBW-map with coefficients is now defined as the composition
\[
\calJ(\calL;\calE)\stackrel{1\otimes -}{\longrightarrow}\widehat{\sym}(\calL^\vee)\otimes_{\calO_{X_2}}\calJ(\calL;\calE)\stackrel{\exp{D}}{\longrightarrow}\widehat{\sym}(\calL^\vee)\otimes_{\calO_{X_2}}\calJ(\calL;\calE)\stackrel{\text{ev}_1}{\longrightarrow}\widehat{\sym}(\calL^\vee)\otimes\calE
\]
The leibniz property of $D$ gives property \ref{JL-morph}. The $\calO_{X_1}$ linearity and the fact the it is an isomorphism follow in the same way as in the proof of 
\ref{dual-pbw}. 
\end{proof}

\begin{corollary}
Using the notation from the previous theorem, any pair $(\nabla^{\calL},\nabla^{\calE})$ of $\calL$-connections on $\calL$ and $\calE$ induces
an isomorphism
\[
j^{*}_{\nabla^{\calL},\nabla^{\calE}}:\sym_{\calO_X}(\calL)\otimes_{\calO_X}{\rm End}(\calE)
\to \calU(\calL;\calE)\
\]
\end{corollary}
\begin{proof}
This follows from $\sym(\calL)\otimes{\rm End}(\calE)\cong\Hom^{\rm cont}(\widehat{\sym}(\calL^{\vee})\otimes\calE,\calE)$ and the fact that $j_{\nabla^\calL,\nabla^\calE}$ is $\calO_{X_1}$-linear.
\end{proof}
There is also a noncommutative, twisted PBW-theorem, which we describe now. 
\begin{theorem}
The data of $\calL$ connections $\nabla^\calL$ and $\nabla^\calE$ on $\calL$ and $\calE$ induces an isomorphism of sheaves of algebras:
\[
\calJ(\calL)\otimes_{\calO_{X_2}} \calU(\calL;\calE)\longrightarrow \Diff_{\calO_X}(\widehat{\sym}_{\calO_X}(\calL^\vee)\otimes_{\calO_X}\calE).
\]
\end{theorem}
\begin{proof}
Similar to the untwisted case, there is a canonical morphism 
\begin{align*}
\calJ(\calL)\otimes_{\calO_{X_2}} \calU(\calL;\calE)&\rightarrow \Diff_{\calO_{X_1}}(\calJ(\calL;\calE))\\
(\phi\otimes e\otimes D\otimes f)&\mapsto (\psi\otimes s\mapsto \phi D\cdot_{2}(\alpha_2(f(s))\psi)\otimes e).
\end{align*}
Where $e,s\in \calE$, $f\in \calE^\vee$, $\phi,\psi\in \calJ(\calL)$, $D\in \calU(\calL)$ and $D$ acts on $\calJ(\calL)$ using $\nabla^{(2)}$. Recall that
differential operators of order $k$ on the $\calJ(\calL)$-module $\calJ(\calL;\calE)$ are endomorphisms of $\calJ(\calL;\calE)$ satisfying $[\phi_0,[\ldots,[\phi_k,d]\ldots]=0$
where $\phi_i\in\calJ(\calL)$ are considered as endomorphisms of $\calJ(\calL;\calE)$. It is straightforward to check that the morphism above is well-defined and has the range that is indicated. 
Moreover, it respects the filtrations from $\calU(\calL)$ and $\calJ(\calL)$ on the LHS and the degree of differential operators and the $\calJ(\calL)$ filtration on the
RHS. The fact that it is an isomorphism follows from the same argument as in the untwisted case. 
Using now that $ \Diff(\widehat{\sym}(\calL^\vee)\otimes \calE)\cong\Diff_{\calO_{X_1}}(\calJ(\calL;\calE))$ by the morphism $D\mapsto j_{\nabla^\calL,\nabla^\calE}\circ D\circ j^{-1}_{\nabla^\calL,\nabla^\calE}$, the claim follows.
\end{proof}

\begin{remark}\label{twisted-remark}
Proposition \ref{depconn} has a straightforward generalization to the twisted case.
\end{remark}

\section{The character map}
\label{sc}
\subsection{The set up}\label{set up}
Cyclic homology of algebras forms a fundamental part of noncommutative geometry. We have summarized the basic definitions and complexes in Appendix \ref{cyclic}. 
In \cite{weibel}, Weibel generalized this theory to a homology theory for schemes by sheafifying the presheaf of chain complexes 
\[
U\mapsto (CC^{W}_{\bullet}(\calO_{X}(U)), b+uB).
\]
given by the cyclic chain complex of the associated locally ringed space $(X,\calO_{X})$. 
The usual Hochschild--Kostant--Rosenberg map
\begin{equation}
\label{hkr}
HKR: f_{0}\otimes\ldots\otimes f_{k}\mapsto f_{0}df_{1}\wedge\ldots\wedge df_{k},\quad f_{i}\in\calO_{X}(U),~i=1,\ldots,k,
\end{equation}
sheafifies to define a morphism $(C_{\bullet}(\calO_{X}),b,B)\mapsto (\Omega^{\bullet}_{\calO_{X}},0,d)$ of mixed complexes of sheaves, 
where $\Omega^{\bullet}_{\calO_{X}}$ denotes the sheaf of K\"ahler differentials associated to the structure sheaf $\calO_{X}$.
\begin{remark}
When $X$ is a manifold, it is customary to complete the sheaf of cyclic chains by taking the germs or jets of functions around the 
diagonal in $X^{\times (k+1)}$. We do not consider such completions here, neither are we concerned with the question whether
the HKR-map is a quasi-isomorphism. However, it is not difficult to prove that in this case our index theorem extends to the completed chain
complexes.
\end{remark}

Likewise, for a Lie algebroid $(\calL,\calO_{X})$, the assigment
\[
U\mapsto (CC^{W}_{\bullet}(\calU(\calL(U))), b+uB),
\]
forms a presheaf of chain complexes. The associated sheaf of chain complexes is simply denoted by $(CC_{\bullet}(\calU(\calL)), b+uB)$.

The aim of this section is to construct a map between this complex of sheaves and the de Rham complex of sheaves $(\Omega^{\bullet}_{\calL},d_{\calL})$. 
It is the analogue of the so-called ``trace density'' \cite{bnt} or ``character map'' \cite[Theorem 3.9]{ppt} in deformation quantization. In the setting of the paper,
the proper formulation is in terms of the derived category $\calD(X)$ of sheaves on $X$. To state the precise theorem, we define, $\Omega^{\bullet}_{\calL,W}:=\Omega^{\bullet}_{\calL}[[u]]\otimes_{\gf[u]}W$, for any $\gf[u]$-module
$W$.
In the language of derived categories, the main theorem of this section then reads as follows:
\begin{theorem}
\label{character-map}
Let $(\calL,\calO_{X})$ be a Lie algebroid of locally constant rank $r$. There exists a canonical morphism
\[
\Phi^{W}\in\Hom_{\calD(X)}\left(\left(CC^{W}_{\bullet}(\calU(\calL)),b+uB\right),\left(\Omega^{-\bullet}_{\calL,W}[2r],d_{\calL}\right)\right).
\]
This morphism is compatible with SBI-sequence.
\end{theorem}

\subsection{The Fedosov resolution}
In this section we construct the analogue of the Fedosov resolution \cite{fedosov} for $\calU(\calL)$. 
In the smooth category, such a resolution was constructed in \cite{nw} from the point of view of formal deformation quantization of the dual 
of a Lie algebroid equipped with the Lie--Poisson bracket. Here we approach the resolution from the other side, using the PBW theorem of the previous section.
We assume for the moment that $\calL$ admits an $\calL$-connection. The construction of the sheaf of algebras \eqref{formal-diff}
only involved the {\em second} $\calU(\calL)$-module structure \eqref{module2} on $\calJ(\calL)$. Now we use
the {\em first} module structure \eqref{module1} to define
\[
D\cdot(\phi\otimes E):=D\cdot_{1}\phi\otimes E,\quad D,E\in\calU(\calL),~\phi\in\calJ(\calL).
\]
The flat connection $\nabla^{(1)}$ on $\calJ(\calL)$ induces a flat connection on $\Diff_{\calO_{X_1}}(\calJ(\calL))$ by commuting, which can 
be written as $\nabla^{(1)}\otimes 1$. Taking the associated Chevalley--Eilenberg complex;
\[
\Omega^\bullet_{\calL}(\Diff_{\calO_{X_1}}(\calJ(\calL))):=\left(\calJ(\calL)\otimes_{\calO_{X_2}}\calU(\calL)\right)\otimes_{\calO_{X_1}}\bigwedge^\bullet_{\calO_X}\calL^\vee
\]
with differential written as $d_{CE}$, we have:
\begin{lemma}
\label{lemdga}
The Chevalley--Eilenberg complex $\left(\Omega_{\calL_1}^\bullet\left(\Diff_{\calO_{X_1}}(\calJ(\calL))\right),d_{CE}\right)$ is a sheaf of differential graded algebras.
\end{lemma}
\begin{proof}
By construction, $\Omega_{\calL_1}^\bullet\left(\Diff(\calJ(\calL))\right)$ is both a graded algebra (combining the multiplication \eqref{mult-fd} with the wedge product) and a cochain complex.
Therefore, the only thing to check is the compatibility between  both structures. 
Using the fact that both $\calU(\calL)$-module structures \eqref{module1} and \eqref{module2} on $\calJ(\calL)$ commute, one has that
\begin{align*}
D\cdot((\phi\otimes E)\cdot(\psi\otimes F))&=D(\phi (E_{(1)}\cdot_{2}\psi)\otimes E_{(2)}F)\\
&=D(\phi (E_{(1)}\cdot_{2}\psi))\otimes E_{(2)}F\\
&=(D_{(1)}\cdot_{1}\phi)(D_{(2)}\cdot_{1}(E_{(1)}\cdot_{2}\psi))\otimes E_{(2)}F\\
&=(D_{(1)}\cdot_{1}\phi)(E_{(1)}\cdot_{2}(D_{(2)}\cdot_{1}\psi))\otimes E_{(2)}F\\
&=(D_{(1)}\cdot_{1}\phi \otimes E)\cdot ((D_{(2)}\cdot_{1}\psi)\otimes F\\
&=(D_{(1)}\cdot (\phi \otimes E))\cdot (D_{(2)}\cdot(\psi\otimes F)).
\end{align*}
This means that the $\calL$-connection $\nabla^{(1)}\otimes 1$ inducing this $\calU(\calL)$-module structure is a derivation for the 
algebra structure \eqref{mult-fd}. This proves the Lemma.
\end{proof}

\begin{proposition}
The map $\calU(\calL)\to \Omega_{\calL_1}^\bullet(\Diff_{\calO_{X_1}}(\calJ(\calL)))$ is a quasi-isomorphism of
complexes of sheaves.
\end{proposition}
\begin{proof}
In \cite[Prop. 4.2.4]{cvdb} it is proved that the inclusion $\calO_X\stackrel{\alpha_2}{\longrightarrow} \Omega_{\calL_1}^\bullet(\calJ(\calL))$ is 
a quasi-isomorphism of complexes of sheaves. By exactly the same reasoning, namely considering the associated graded complexes
of the filtration of $\calJ(\calL)$ and observing that the associated map is a quasi-isomorphism, the lemma is proved.
\end{proof}
\begin{remark}
In the smooth category, it is a well-known fact that the universal enveloping algebra of a Lie algebroid is isomorphic to the algebra of invariant differential operators on 
an integrating Lie groupoid. The preceeding Proposition should be viewed as an algebraic counterpart of that statement, with the integrating Lie groupoid being replaced 
by the Hopf algebroid of jets, and instead of taking invariants, we consider the cohomology of $\Omega_{\calL_1}^\bullet(\Diff_{\calO_{X_1}}(\calJ(\calL)))$.
\end{remark}
We now proceed analogous to the Fedosov procedure in deformation quantization: on the bundle $\calW(\calL)$ there are two connections: one given by the differential induced 
by the $\calL$-connection $\nabla$, the other given by the Chevalley--Eilenberg operator $d_{CE}$, mapped to $\calW(\calL)$ via the noncommutative PBW map of Theorem \ref{ncpbw}.
More precisely, Theorem \ref{ncpbw} gives an isomorphism of sheaves of graded algebras:
\[
j_\nabla: \Omega^\bullet_{\calL_1}(\Diff_{\calO_{X_1}}(\calJ(\calL)))\stackrel{\cong}{\longrightarrow}\Omega^{\bullet}_{\calL}(\calW(\calL)):=\calW(\calL)\otimes\bigwedge^\bullet\calL^\vee.
\]
This isomorphism induces a DGA structure on the right hand side. On the other hand, the $\calL$-connection $\nabla^\calL$ induces a natural degree $1$-
operator on the algebra of differential forms with values in the Weyl algebra, denoted by 
$\nabla^\calL:\Omega_\calL^\bullet(\calW(\calL))\to\Omega_\calL^{\bullet+1}(\calW(\calL))$.
Notice that this operator will not define a differential, i.e., $\nabla^2\not =0$, unless $\nabla$ happened to be flat. 
Let us explicitly compare the two operators on $\Omega^{\bullet}_{\calL}(\calW(\calL))$ by defining
\begin{equation}
\label{fedosov-connection}
\alpha(\nabla):=j_\nabla\circ d_{CE}\circ j^{-1}_\nabla-\nabla.
\end{equation}
The depence of $\alpha$ on $\nabla$ will be indicated when necessary by writing $\alpha(\nabla)$. A priori, $\alpha(\nabla)$ is an element
of $\Omega^{1}_{\calL}(\End(\calW(\calL)))$, but since $d_{CE}$ and $\nabla$ both act by derivations $\alpha(\nabla)\in\Omega^{1}_{\calL}(\Der(\calW(\calL)))$.
Since every derivation on the Weyl algebra is inner, c.f.\ \eqref{derw}, we can lift $\alpha(\nabla)$ to an element $A(\nabla)\in\Omega^{1}_{\calL}(\calW(\calL))$.
We can in fact choose a canonical lift as follows: restrict $d_{CE}$ and $\nabla$ to $\calJ(\calL)$, resp.\ $\widehat{\sym}(\calL^{\vee})$. Then the difference
\begin{equation}
\label{fedcon}
A(\nabla)=j_\nabla\circ d_{CE}\circ j_\nabla^{-1}-\nabla\in\Omega^{1}_{\calL}(\Der(\widehat{\sym}(\calL^{\vee}))),
\end{equation}
is, by definition of the Weyl algebra, an element in $\Omega^{1}_{\calL}(\calW)$. We claim that $\ad(A(\nabla))=\alpha(\nabla).$
This follows from the fact that $\nabla$ and $d_{CE}$ on $\calW(\calL)=\Diff(\widehat{\sym}(\calL))$ are defined as the commutator with a differential operator,
e.g. for $D\in \Diff(\widehat{\sym}(\calL))$, $\alpha \in \widehat{\sym}(\calL)$ and $X\in \calL$:
\[
\nabla_X(D)(\alpha):=\nabla_X(D(\alpha))-D(\nabla_X(\alpha)).
\]
The crucial property of $A(\nabla)$ is that it satisfies the Maurer Cartan equation.
\begin{lemma}\label{lemma-mc}
Let $A(\nabla)$ be as defined in \eqref{fedcon}. Then:
\begin{equation}
 R(\nabla)+\nabla A(\nabla)+\frac{1}{2}[A(\nabla),A(\nabla)]=0
\end{equation}
\end{lemma}
\begin{proof}
This follows from the fact that $\nabla+A(\nabla)$ is a differential. Note that $R(\nabla)$ is the curvature of $\nabla$.
\end{proof}

\subsection{The character map}
As mentioned in \ref{set up}, the Hochschild and cyclic (periodic, negative) complexes in the setting of a locally free Lie algebroid over a ringed space
are complexes of sheaves. The definition of the Character map depends on the choice of an $\calL$ connection, and such a connection only exists locally. 
We therefore choose a covering $\{U_{i}\}_{i\in I}$ of $X$ by opens $U_{i}$ over which we do have connections $\nabla_{i}^{\calL}$. Associated to the covering is the 
\v{C}ech groupoid 
\[
X_\calU:=\coprod_{i,j\in I}U_{i}\cap U_{j}\rightrightarrows \coprod_{i\in I}U_{i}.
\]
On its $k$-th nerve $X_{\calU,k}:=\coprod_{i_0,\ldots,i_k\in I}U_{i_0}\cap\ldots\cap U_{i_k}$, we have the family of connections 
\[
\nabla^{\rm aff}_k:=\nabla_{i_0}+\sum_{j=1}^kt_j(\nabla_{i_j}-\nabla_{i_0}),\quad 0\leq t_1\leq \ldots\leq t_k \leq 1,
\]
parametrized by the $k$-simplex $\Delta^k$. In the following, we simply write $\underline{t}=(t_1,\ldots,t_k)\in\Delta^k$.
This ``simplicial connection'' $\nabla^{\rm aff}_k$ is used to construct the following elements:
\begin{itemize}
\item  the family of connection forms $A^{\rm aff}_k(\underline{t})$ via \eqref{fedcon},
\item the derivations $\theta_j(\underline{t}),~j=1,\ldots,k$ by applying Proposition \ref{depconn} to $d\nabla^{\rm aff}_k(\underline{t})\slash dt_j$.
\end{itemize}
These two data appear in the definition of the character map: given a chain $D=D_0\otimes\ldots\otimes D_l\in CC^W_l(\calU(\calL))$, 
we define
\begin{equation}\label{character}
\Phi^{W,w}_{\calU,\nabla}(D)|_{U_{i_0\ldots i_k}}:=(-1)^\alpha\int_{\Delta^k}\iota_{\theta_{i_1}}\ldots\iota_{\theta_{i_k}}\tau^{w}_{2r}(1\otimes \exp(\wedge A^{\rm aff}_k(\underline{t}))\times D)d\underline{t}.
\end{equation}
Here, the cocycle $\tau^w_{2r}$ is defined in \ref{LRR}. It is a cocyle for the Weyl algebra $\calW(\calL)$, thus in fact we have to apply
$j_{\nabla}$ to evaluate $\tau_{2r}^w$. The element $\exp(\wedge A)$ is an $\calL$-form with values in the Weyl algebra such that $\exp(\wedge A)(X_1,\ldots,X_n)=A(X_1)\wedge\ldots\wedge A(X_n)$. Finally, the sign is given by $\alpha={\sum_{i=1}^k i+\sum_{j=1}^l j}$.

For varying $k$ and $l$, we want to view the equation \eqref{character} as a map 
\[
\Phi^{W,w}_{\calU,\nabla}: \left(CC^W_{\bullet}(\calU(\calL)),b+uB\right)\longrightarrow \left(\Omega_{\calL,W}^{-\bullet}[2r](X_{\calU,\bullet}),\delta+(-1)^{*}d_{\calL}\right),
\]
where $\delta$ denotes the \v{C}ech differential on the nerve of the \v{C}ech groupoid. Then the first main step towards the proof of Theorem \ref{character-map} is given by 
the following
\begin{proposition} \label{mapofcomplexes}
The map $\Phi^{W,w}_{\calU,\nabla}$ defined in equation \eqref{character} is a morphism  of complexes.
\end{proposition}
Given a chain $D_0\otimes\ldots\otimes D_l$, abbreviated by $D$,
over $U_{i_0\ldots i_k}$, we have to prove that, recalling \ref{character},
\begin{equation}\label{commdiff}
\begin{split}
\Phi^{k}_{l-1,2s}(b(D))+\Phi_{l+1,2s+2}^k(B(D))=\delta(\Phi_{l,2s}^{k-1}(D))+(-1)^kd_\calL(\Phi_{l,2s}^k(D)).
\end{split}
\end{equation}
 This equality should hold in $\Omega_\calL^{2s-k-l+1}|_{U_{i_0\ldots i_k}}$. 
 The subscript $2s$ indicates which part of $\tau_w$ is used is used in the definition \eqref{character} of $\Phi$, c.f. the expansion in \ref{lrr}.
\begin{remark}
The most natural choice for $w\in W$ is $1$. However, in the case of negative cyclic homology this
choice leads to problems because the cocylce $\tau_{2r}^w$ contains $u$'s, which implies that this map would not
be well defined. Setting $w=u^r$ circumvents this problem.
\end{remark}
Before we start with the proof of Proposition \ref{mapofcomplexes}, we single out two Lemmas needed its course. 
First, let us simply write $j^{\rm aff}_k$ for the family of PBW-isomorphisms $j_{\nabla^{\rm aff}_k(\underline{t})}$  induced by
the connection $\nabla_k^{\rm aff}(\underline{t})$ parametrized by $\underline{t}\in\Delta^k$. 
The following Lemma is an easy consequence of proposition \ref{depconn}. 
\begin{lemma}\label{depconn2}
The following equations hold true:
\begin{align*}
\partial_{t_i}j^{\rm aff}_k&=\ad_{\theta_i}(j^{\rm aff}_k)\\
\partial_{t_i} A^{\rm aff}_k&=\ad_{\theta_i}( A^{\rm aff}_k)-\nabla^{\rm aff}_{k}\theta_i-\partial_{t_i}\nabla^{\rm aff}_k\\
\partial_{t_i}\theta_j-\partial_{t_j}\theta_i&=[\theta_i,\theta_j].
\end{align*}
\end{lemma}
\begin{proof}
This follows from straightforward computations. Let $K\in \Diff(\widehat{\sym}(\calL^\vee)$, then
\begin{align*}
\partial_{t_i}( j^{\rm aff}_k(K))=&\partial_{t_i}(j^{\rm aff}_k\circ K\circ  (j^{\rm aff}_k)^{-1})\\
=&\ad_{\theta_i}\circ  j^{\rm aff}_k(K).
\end{align*}
Next, we have
\begin{align*}
\partial_{t_i}(A^{\rm aff}_k+\nabla^{\rm aff}_k)=&\partial_{t_i}(j^{\rm aff}_k\circ d_{CE} \circ  (j^{\rm aff}_k)^{-1})\\
=&\ad_{\theta_i}(j^{\rm aff}_k\circ d_{CE} \circ  (j^{\rm aff}_k)^{-1})\\
=&\ad_{\theta_i}(A^{\rm aff}_k)-\nabla^{\rm aff}_k \theta_i.
\end{align*}
Finally we have, for the family of (commutative) PBW isomorphisms $j^{\rm aff}_k$;
\[
\partial_{t_i}\partial_{t_j} j^{\rm aff}_k=(\partial_{t_i}\theta_j+\theta_j\theta_i) j^{\rm aff}_k
\]
which gives, together with the same equality where $i$ and $j$ are interchanged, the required expression.
\end{proof}

Next, we  need some equalities involving the Maurer--Cartan element $A(\nabla)$ in the sheaf of unital, graded algebras $\Omega^\bullet(\calW(\calL))$.
Recall that $A(\nabla)$ satisfies the Maurer--Cartan equation of Lemma \ref{lemma-mc}.
Set $(A)^k=\frac{1}{k!}1\otimes A\wedge\ldots\wedge A$,
which is part of $1\otimes \exp(\wedge A)$,
and consult \ref{cyclic} for the definitions of cyclic homology and the shuffle product. 
\begin{lemma}\label{shuffle}
For all
$a=a_0\otimes\dots\otimes a_p \in C_p(\Omega^\bullet(\calW(\calL)))$  one has the equalities
\begin{equation*}
\begin{split}
b((A)^k\times a)=&b((A)^k)\times a+(-1)^k(A)^k\times b(a)\\[5pt]
+&(-1)^k\sum_i (A)^{k-1}\times (a_0\otimes\ldots\otimes [A,a_i]\otimes\ldots a_l),\\[5pt]
b((A)^k)=&\nabla((A)^{k-1}),\\[5pt]
B((A)^k\times a)=&(-1)^k(A)^k\times B(a),\\
\end{split}
\end{equation*}
where $[a,b]$ denotes the graded commutator.
\end{lemma}
\begin{proof}
Entirely analogous to  \cite[Lemma 2.5 \& 2.6]{ef} and \cite[Lemma 3.3]{ppt}.
\end{proof}

\begin{proof}[Proof of Proposition \ref{mapofcomplexes}]
On the intersections $U_{i_0\ldots \hat{i}_j\ldots i_k}$ the coordinates $t_1,\ldots \hat{t}_j,\ldots t_k$ for $j\geq 1$ and $t_2,\ldots ,t_k$ for $j=0$,
and families of connections $\nabla_0+\sum_{i\neq j}t_i(\nabla_i-\nabla_0)$ and $\nabla_1+\sum_{i=2}^kt_2(\nabla_i-\nabla_1)$, together with Stokes' theorem imply
\[
\begin{split}
\delta(\Phi_{l,2s}^{k-1}(D))|_{U_{i_0\ldots i_k}}=&\sum_{j=0}^k(-1)^j\Phi_{l,2s}^{k-1}(D)|_{U_{i_0\ldots \hat{i}_j\ldots i_k}}\\
=&\sum_{j=0}^k(-1)^j\int_{\partial_j\Delta^k}
\left(\sum_{i=1}^k\iota_{\theta_{1}}\ldots \hat{\iota}_{\theta_{i}}\ldots \iota_{\theta_{k}}\tau_{2s}(A\times D)dt_1\ldots \hat{dt_i}\ldots dt_k\right)\\
=&\int_{\Delta^k}d_{\Delta^k}\left(\sum_{j=1}^k\iota_{\theta_{1}}\ldots \hat{\iota}_{\theta_{j}}\ldots \iota_{\theta_{k}}\tau_{2s}(A\times D)dt_1\ldots \hat{dt_i}\ldots dt_k\right)\\
=&\int_{\Delta^k}\sum_{j=1}^k(-1)^{j+1}\partial_{t_j}\left(\iota_{\theta_{1}}\ldots \hat{\iota}_{\theta_{j}}\ldots \iota_{\theta_{k}}\tau_{2s}(A\times D)\right)dt_1\ldots dt_k
\end{split}
\]

Note that $A$ stands for $(A)_{j}$ where $j$ follows from the formulas.
Hereafter we will be concerned with rewriting the integrand in the last expression.

 We omit the term $dt_1\dots dt_k$. First we split the integrand into two terms:
\begin{align*}
&\sum_{j=1}^k(-1)^{j+1}\iota_{\theta_{1}}\ldots \hat{\iota}_{\theta_{j}}\ldots \iota_{\theta_{k}}\partial_{t_j}\left(\tau_{2s}(A\times D)\right)\quad\quad &(*)\\
+&\sum_{j\neq i}(-1)^{j+1}\iota_{\theta_{1}}\ldots \iota_{\partial_{t_j}\theta_i}\ldots \hat{\iota}_{\theta_{j}}\ldots \iota_{\theta_{k}}\partial_{t_j}\tau_{2s}(A\times D)\quad\quad &(**)
\end{align*}
An application of the relations for the derivatives of $j^{\rm aff}_k$ and $A^{\rm aff}_k$, c.f.\  lemma \ref{depconn2}, gives rise to
\[
\begin{split}
(*)=&\sum_j(-1)^{j+1}\iota_{\theta_{1}}\ldots \hat{\iota}_{\theta_{j}}\ldots \iota_{\theta_{k}}\tau_{2s}(L_{\theta_j}(A\times D))\\
+&\sum_j(-1)^{j+1}\sum_i\iota_{\theta_{1}}\ldots \hat{\iota}_{\theta_{j}}\ldots \iota_{\theta_{k}}(\tau_{2s}((1\otimes A\otimes\ldots \otimes\underbrace{-\nabla\theta_j}_i \otimes \ldots \otimes A)\times D)\\
=&\sum_j(-1)^{j+1}L_{\theta_j}\iota_{\theta_{1}}\ldots \hat{\iota}_{\theta_{j}}\ldots \iota_{\theta_{k}}\tau_{2s}((A)^{2s-k-l+1}\times D))\\
+&\sum_j(-1)^{j+1}\iota_{\nabla\theta_j}\iota_{\theta_{1}}\ldots \hat{\iota}_{\theta_{j}}\ldots \iota_{\theta_{k}}\tau_{2s}((A)^{2s-k-l}\times D))
\end{split}
\]
where we ommited the term involving $\partial_{t_i}\nabla^{\rm aff}_k$ because of the fact that $\tau$ is basic, c.f.\ equation (\ref{basic}).
The relation $\partial_{t_i}\theta_j-\partial_{t_j}\theta_i=[\theta_i,\theta_j]$ can be used to show that
\[
\begin{split}
(**)=&\sum_{j\neq i}(-1)^{j+1}\iota_{\theta_{1}}\ldots \iota_{\partial_{t_j}\theta_i}\ldots \hat{\iota}_{\theta_{j}}\ldots \iota_{\theta_{k}}\tau_{2s}(A\times D)\\
=&\sum_{i<j}(-1)^{j}\iota_{\theta_{1}}\ldots \iota_{[\theta_i,\theta_j]}\ldots \hat{\iota}_{\theta_{j}}\ldots \iota_{\theta_{k}}\tau_{2s}(A\times D).
\end{split}
\]
Moreover, we have an equality derived from $[L_a,\iota_b]=\iota_{[b,a]}$:
\[
\begin{split}
\sum_j(-1)^{j+1}\iota_{\theta_1}\ldots L_{\theta_j}\ldots\iota_{\theta_k}\tau_{2s}(A\times D)
=&\sum_j(-1)^{j+1}L_{\theta_j}\ldots\hat{\iota}_{\theta_j}\ldots\iota_{\theta_k}\tau_{2s}(A\times D)\\
+&\sum_{i<j}(-1)^{j}\iota_{\theta_{1}}\ldots \iota_{[\theta_i,\theta_j]}\ldots \hat{\iota}_{\theta_{j}}\ldots \iota_{\theta_{k}}\tau_{2s}(A\times D).
\end{split}
\]
Collecting the previous formulas amounts to 
\[
\begin{split}
(*)+(**)=&\sum_j(-1)^{j+1}\iota_{\nabla\theta_j}\iota_{\theta_{1}}\ldots \hat{\iota}_{\theta_{j}}\ldots \iota_{\theta_{k}}\tau_{2s}((A)^{2s-k-l}\times D))\\
+&\sum_j(-1)^{j+1}\iota_{\theta_1}\ldots L_{\theta_j}\ldots\iota_{\theta_k}\tau_{2s}((A)^{2s-k-l+1}\times D).
\end{split}
\]
Applying the relation $b\iota_a+\iota_ab=L_a$ several times gives
\[
\sum_j(-1)^{j+1}\iota_{\theta_1}\ldots L_{\theta_j}\ldots\iota_{\theta_k}=b\iota_{\theta_1}\ldots\iota_{\theta_k}+(-1)^{k+1}\iota_{\theta_1}\ldots\iota_{\theta_k}b.
\]
Moreover, we have the following two equalities, obtained from lemma \ref{shuffle} and a simple relation relating the shuffle operation to the insertion operator:
\[
\begin{split}
b((A)^{2s-k-l+1}\times D)=&\nabla(A)^{2s-k-l}\times D-(-1)^{k+l}(A)^{2s-k-l+1}\times b(D)\\[8pt]
-&(-1)^{k+l}\sum_{i}(A)^{2s-k-l}\times D_0\otimes\ldots\otimes [A,D_i]\otimes\ldots\otimes D_l\\[8pt]
\iota_{\theta_1}\ldots\iota_{\theta_k}b\tau_{2s}(A\times D)=&(-1)^k\tau_{2s+2}(B((\theta_k)\times\ldots\times (\theta_1)\times A\times D))\\[8pt]
=&(-1)^{k+l+1}\tau_{2s+2}((\theta_k)\times\ldots\times (\theta_1)\times A\times B(D))\\[8pt]
=&(1)^{l+1}\iota_{\theta_1}\ldots\iota_{\theta_k}\tau_{2s+2}(A\times B(D))
\end{split}
\]
Using that $\nabla D_i+[A,D_i]=0$, we obtain
\[
\begin{split}
(*)+(**)
=&\sum_j\iota_{\theta_{1}}\ldots \iota_{\nabla\theta_j}\ldots \iota_{\theta_{k}}\tau_{2s}((A)^{2s-k-l}\times D))\\
+&\iota_{\theta_1}\ldots\iota_{\theta_k}\left(\tau_{2s}(\nabla (A)^{2s-k-l}\times D)+(-1)^{k+l}\tau_{2s}((A)^{2s-k-l}\times\nabla(D))\right)\\[8pt]
+&(-1)^{k+l+1}\iota_{\theta_1}\ldots\iota_{\theta_k}\tau_{2s}((A)^{2s-k-l+1}\times b(D))\\[8pt]
+&(-1)^{k+l+1}\iota_{\theta_1}\ldots\iota_{\theta_k}\tau_{2s+2}((A)^{2s-k-l+1}\times B(D))\\[8pt]
=&d_\calL\left(\iota_{\theta_1}\ldots\iota_{\theta_k}\tau_{2s}(A\times D)\right)+(-1)^{k+l+1}\iota_{\theta_1}\ldots\iota_{\theta_k}\tau_{2s}(A\times b(D))\\[8pt]
+&(-1)^{k+l}\iota_{\theta_1}\ldots\iota_{\theta_k}\tau_{2s+2}(A\times B(D))
\end{split}
\]
Which is equivalent to \ref{commdiff}. Of course, the sign $\alpha$ from the definition \ref{character} has to be taken into account.
\end{proof}
We are ready to give the proof of the main theorem:
\begin{proof}[Proof of Theorem \ref{character-map}]
Recall that the derived category of sheaves is constructed out of the category of chain complexes of sheaves by 1) identifying homotopic chain maps (i.e., going over to the homotopy 
category and 2) localization with respect to quasi-isomorphisms. For any covering, the natural map from the Lie algebroid chain complex $(\Omega^{-\bullet}_{\calL,W},d_\calL)$ to its  
\v{C}ech resolution $\check{C}(X_\calU,\Omega^{-\bullet}_{\calL,W})$ is a quasi-isomorphism, so by Proposition \ref{mapofcomplexes} the character map $\Phi^{W,w}_{\calU,\nabla}$ defines a 
morphism in the derived category as stated in Theorem \ref{character-map}. It therefore remains to be shown that this map is independent of the covering and the compatible 
connections chosen. 

Suppose we have two such choices $(\calU,\nabla_\calU)$ and $(\calV,\nabla_\calV)$ of coverings and compatible families of connections. We can merge the two coverings
$\calU$ and $\calV$ into one $\calU\coprod\calV$. The associated \v{C}ech groupoid $X_{\calU\coprod\calV}$ decomposes as
\[
X_\calU\coprod X_{\calU\cap\calV} \coprod X_\calV\rightrightarrows \calU\coprod\calV,
\]
where $X_{\calU\cap\calV}$ consists of all intersections of elements in $\calU$ with those of $\calV$. We therefore find $X_\calU$ and $X_\calV$ as two subgroupoids. 
Let us write the character map 
\[
\Phi^{W,w}_{\calU\coprod\calV,\nabla_{\calU\coprod\calV}}=\Phi^{W,w}_{\calU,\nabla_\calU}+\Phi^{W,w}_{\calU\cap\calV,\nabla_{\calU\cap\calV}}+\Phi^{W,w}_{\calV,\nabla_\calV},
\]
corresponding to the induced decomposition of the simplicial nerve of $X_{\calU\coprod\calV}$, where $\Phi^{W,w}_{\calU,\nabla_\calU}$ corresponds to the subgroupoid $X_\calU$, 
and $\Phi^{W,w}_{\calV,\nabla_\calV}$ to $X_\calV$: the remaining terms are written as $\Phi^{W,w}_{\calU\cap\calV,\nabla_{\calU\cap\calV}}$. We now claim that exactly
this last term $\Phi^{W,w}_{\calU\cap\calV,\nabla_{\calU\cap\calV}}$ gives a homotopy between $\Phi^{W,w}_{\calU,\nabla_\calU}$ and $\Phi^{W,w}_{\calV,\nabla_\calV}$ composed 
with the obvious map to $\check{C}(X_{\calU\coprod\calV},\Omega^{-\bullet}_{\calL,W})$. This is seen by application of equation \eqref{commdiff} satisfied by $\Phi^{W,w}_{\calU
\coprod\calV,\nabla_{\calU\coprod\calV}}$.
\end{proof}

\begin{example}
As an example of this construction, let us consider the Weyl algebra $W_{n}$ introduced in Appendix \ref{weyl}.
Indeed, $W_{n}$ is just the universal enveloping algebra of the Lie--Rinehart algebra $(L_{n},\calO_{n})$, where $L_{n}=\Der(\calO_{n})=\calO_{n}\left<\partial/\partial x^{1},\ldots,\partial/\partial x^{n}\right>$. We denote the dual basis of $L_{n}^\vee$ by $\{dx^i\},\; i\in\{1,\ldots n\}$.  
This Lie--Rinehart algebra admits a canonical flat $L_{n}$-connection defined by
\[
\nabla_{\frac{\partial}{\partial x^i}}\left(\frac{\partial}{\partial x^j}\right)=0,\quad\quad ~i,j=1,\ldots,n.
\]
It is straightforward to deduce that $A$ is given in this case by $\sum_i 1\otimes \frac{\partial}{\partial x^i}\otimes dx^i$,
we refer to lemma \ref{lemma-fc} for more explanation. The Hochschild and cyclic (co)homology of $W_n$ are known, see \ref{LRR},
and the image of the only nontrivial generator $c_{2n}$ defined in \ref{LRR} is, using equation \ref{normalization}, equal to 1, 
the only generator of the de Rham cohomology. We conclude that in this case the character map $\Phi$ induces an isomorphism for Hochschild and cyclic homology.
\end{example}

\begin{remark}\label{twisted-remark2}
Continuing remark \ref{twisted-remark} at the end of the last section, the constructions in this section admit straightforward generalizations to the twisted case. 
Let us elaborate a bit. The element $j_{\nabla^{\calL},\nabla^{\calE}}\circ d_{\rm ce}\circ j_{\nabla^{\calL},\nabla^{\calE}}^{-1}-(\nabla^\calL\times 1+1\times \nabla^\calE)= A\in \Omega_\calL^1(\Der(\sym(\calL^\vee)\otimes \calE))$ satisfies the following Maurer--Cartan equation:
\[
\frac{1}{2}[A,A]+(\nabla^{\calL,\calE})A+R(\nabla^\calL)+R(\nabla^\calE)=0.
\]
When either $\calL$ or $\calE$ does not admit a global connection, we again need an open cover such that connections exist on elements of the cover. 
We form affine combinations of connections on intersections and we can consider the derivatives of $j_{\nabla_{\underline{t}}^\calL,\nabla_{\underline{t}}^\calE}$, the induced 
noncommutative PBW theorem $A$ and $\theta$. The results are similar to proposition \ref{depconn2}. Explicitly, we have:
\[
\theta_{\underline{t}}=\frac{1-\exp(\nabla^{\calL,\calE}_{\underline{t}})}{\nabla^{\calL,\calE}_{\underline{t}}}(\gamma_{\calL}+\gamma_{\calE}).
\]
\end{remark}

\section{The index theorem}
\label{sit}
In this section we give a proof of the second part of the main theorem stated in the introduction: the index theorem. In its most general setting, this theorem compares the 
character map with the Hochschild--Kostant--Rosenberg map \eqref{hkr}. The failure of these maps to coincide is given by certain characteristic classes, we therefore start by
introducing these. 

\subsection{Characteristic classes}
\label{cw}
 As before, we denote by $(\calL,\calO_{X})$ a Lie algebroid
over a locally ringed space $(X,\calO_{X})$, and let $\calE$ be a locally free $\calO_{X}$-module of constant finite rank $r$. We start by adapting the usual Chern--Weil construction 
to the case where $\calE$ admits an  $\calL$-connection $\nabla^{\calE}$ with curvature $R(\nabla^{\calE})\in\Omega^{2}_{\calL}({\rm End}(\calE))$ defined by equation \eqref{curv-conn}. Then, for any $GL(\K,r)$-invariant polynomial
$P_{k}:M_{r}(\K)\to\K$ of degree $k$, the form
\[
P_{k}(R(\nabla^{\calE}))\in\Omega^{2k}_{\calL},
\]
is well defined by local triviality of $\calE$, and closed, i.e., $dP_{k}(R(\nabla^{\calE}))=0$. The usual arguments of Chern--Weil theory apply verbatim to show that the induced cohomology class $[P_{k}(R(\nabla^{\calE}))]\in H^{2k}({\calL})$
is independent of the chosen connection. In the algebraic case for Lie--Rinehart algebras this gives the Chern character considered in \cite{maakestad}.
In the smooth setting this version of Chern--Weil theory has been considered in \cite{crainic, fernandes}.

To generalize to the general, i.e., non-projective case, we use simplicial methods.
We therefore need the Lie algebroid associated standard simplices $\Delta^{n}$, this time parameterized by $(t_0,\ldots,t_n)$ satisfying $\sum_it_i=1$. For this, let 
\[
\mathsf{R}_{n}:=\gf[t_{0},\ldots,t_{n}]\slash\left<t_{0}+\ldots+t_{n}=1\right>.
\]
The de Rham complexes $\Omega^{\bullet}_{n}:=(\bigwedge^{\bullet}\mathsf{T}^{\vee}_{n},\mathsf{d}_{n})$ of the associated Lie-Rinehart algebra $\mathsf{T}_{n}:=\Der(\mathsf{R}_{n})$, 
form, for varying $n$, al simplicial differential graded algebra considered in \cite{bg}: any morphism $f:[m]\to[n]$ in the simplicial category $\underline{\mathsf{\Delta}}$, induces a morphism 
$f^{*}:\Omega^{\bullet}_{n}\to\Omega^{\bullet}_{m}$ defined on generators by
\[
f^{*}t_{j}=\sum_{f(i)=j}t_{j},\quad 0\leq j\leq n.
\]
Of course, these are just the polynomial de Rham algebras over the simplices $\Delta^{\bullet}$ and there are ``integration maps'', algebraically defined, 
\[
\int_{\Delta^{n}}:\Omega_{n}^{n}\to\gf,
\]
satisfying all the usual properties, most notably the algebraic analogue of Stokes' theorem.
The Lie algebroid over $\Delta^{n}$ induced by this Lie--Rinehart algebra is denoted by $(\mathcal{T}_{n},\mathsf{R}_{n})$.
With this, we can consider the direct sum Lie algebroid $\calL\oplus\mathcal{T}_{n}$ as in \cite[\S 4.2]{mackenzie} over the product $\Delta^{n}\times\sfG_{n}$ and equip it with the connection
\[
\nabla^{\rm tot}_{n}:=\mathsf{d}_{n}+\sum_{i=0}^{n}t_{i}\nabla^{\calE}_{i}.
\]
The Chern--Weil construction for Lie algebroids is then:
\begin{proposition}
Let $\mathcal{I}_{\bullet}$ be the ring of $GL(\K,r)$-invariant polynomial
$P:M_{r}(\K)\to\K$. 
\begin{itemize}
\item[$i)$]
For $P\in \mathcal{I}_{k}$, the forms
\[
P(\calE,\{\nabla_{i}^{\calE}\})_{n}:=\int_{\Delta^{n}}P(R(\nabla^{\rm tot}_{n}))\in\Omega^{2k-n}_{\calL}(\sfG_{i}),\quad n=0,\ldots, 2k,
\]
satisfy 
\[
d_{\calL} P(\calE,\{\nabla_{i}^{\calE}\})_{n}=\delta P(\calE,\{\nabla_{i}^{\calE}\})_{n-1},
\] 
and therefore forms a \v{C}ech cocycle for the sheaf complex $(\Omega_{\calL}^{\bullet},d)$.
\item[$ii)$] Choosing different coverings and connections produces cohomologous cocycles.
\item[$iii)$] The resulting map
\[
\mathcal{I}^{\bullet}\to H^{2\bullet}_{\rm Lie}(\calL)
\]
is a ring homomorphism.
\end{itemize}
\end{proposition}
\begin{proof}
The first claim follows from an application of Stokes' theorem and the fact that  $P(R(\nabla^{\rm  aff}_{n}))$ is closed w.r.t. the lie algebroid differential $d_{\calL\oplus \mathcal{T}_n}$:
\begin{align*}
d_\calL\left(\int_{\Delta^{n}}P(R(\nabla^{\rm tot}_{n}))\right)=&(-1)^n\int_{\Delta^{n}}d_{\calL\oplus \mathcal{T}_n}\left(P(R(\nabla^{\rm tot}_{n}))\right)+\int_{\Delta^{n}}d_{\mathcal{T}_n}\left(P(R(\nabla^{\rm tot}_{n}))\right)\\
=&\sum_{i=0}^n(-1)^i\int_{\partial_i\Delta^{n}}P(R(\nabla^{\rm tot}_{n})).
\end{align*}
To prove the second claim, note first that in the case of two open coverings, we can look at a
refinement of both and apply a similar argument as in the proof of \ref{character-map} . In the case of different connections, 
i.e. for each open $U_i$ we have two connections $\nabla^1_i$ and $\nabla^2_i$, one forms 
the affine combinations 
\[
\tilde{\nabla}^{\rm tot}_{i_0\ldots i_n}=(1-s)\nabla_{i_0\ldots i_n}^{\rm tot,1}+s\nabla_{i_0\ldots i_k}^{\rm tot,2}
\]
on the bundle $\pi^*\calE_{U_{i_0\ldots i_n}}\times [0,1]\rightarrow U_{i_0\ldots i_n}\times \Delta^n$. 
Defining the homotopy operator 
\[
h: \Omega^k_{\calL\oplus \mathcal{T}_n\oplus \mathcal{T}_1}\rightarrow \Omega_{\calL\oplus\mathcal{T}_n}^{k-1}
\]
by integration along the fiber, one proves, using the usual argument, that
\[
d_\calL(h(P(\tilde{\nabla}^{\rm tot}_{i_0\ldots i_k})))=P(\nabla^{\rm tot,2}_{i_0\ldots i_k})-P(\nabla^{\rm tot,1}_{i_0\ldots i_k}).
\]
For the last claim, we first remark that the product on the \v{C}ech resolution of $\Omega^\bullet_\calL$ is given by the Alexander-Whitney product,
which is graded commutative on cohomology. The Chern-Weil map factorizes through the complex $\oplus_i\Omega^{\bullet}_{\calL\oplus \mathcal{T}_i}(\Delta^n\times\sfG_i)$, c.f. \cite{dupont}, and since the Chern-Weil map to this complex is given by the standard formula and thus a homomoprhism, the result follows.
\end{proof}

In the following, we only need the following two invariant polynomials:
\begin{equation}
\label{cherntodd}
{\rm Ch}(X):={\rm Tr}\left(e^{X}\right),\quad {\rm Td}(X):=\det\left(\frac{X}{1-e^{X}}\right)
\end{equation}
As usual, it is meant that one takes the power series expansion around zero of the analytic functions used in these
definitions, and applies the Chern--Weil construction to each homogeneous component. The resulting characteristic classes
are called the Chern class $\Ch_{\calL}(\calE)$ and the Todd class $\Td_{\calL}(\calE)$.

\subsection{The index theorem}
We are now in a position to prove  the index theorem for Lie algebroids. Recall the character maps $\Phi^{W}_{\calL}$ of Theorem \ref{character-map} mapping --in the derived category-- the sheaves of Hochschild and cyclic complexes to 
the Lie algebroid complexes. Applied to the cycle $1\in C^{W}_{0}(\calU(\calL;\calE))$, we obtain a  class
\[
\Phi^{W}_{\calL}(1)\in H^{ev}({\calL})
\]
in Lie algebroid cohomology
which is canonically defined, i.e., an invariant of the Lie algebroid itself. The main index theorem of this section identifies this class in terms of the characteristic classes of the previous section:
\begin{theorem}
\label{index-theorem}
$\Phi^{W}_{\calU}(1)=\Td_{\calL}(\calL)\ch_{\calL}(\calE)$
\end{theorem}
\begin{proof}
Unraveling the definition of the map $\Phi^{W}_{\calL}$ in the proof of Theorem \ref{character-map}, we see that we have to identify, for any choice $(\{U_{i}\}_{i\in I},\nabla_{i})$ of a covering of $X$ together with an $\calL$-connection over $U_{i}$,  
the differential forms
\[
\int_{\Delta^{k}}\iota_{\theta_{i_{1}}}\ldots\iota_{\theta_{i_{k}}}\tau^{W}_{2r}(1,A,\ldots ,A)dt_{1}\cdots dt_{k}\in\Omega^{2r-k}_{\calL,W}(U_{i_{0},\ldots,i_{k}}).
\]
For a fixed covering but varying depth of intersection, these forms together form a closed cycle in the \v{C}ech-de Rham complex of the Lie algebroid $\calL$. 

Let us now consider the integrand of the expression above. Using the fact that the evaluation morphism to Lie algebra cohomology of Appendix \ref{mor-lie} commutes
with taking insertions $\iota_{\theta_{i}}$, together with the local Riemann--Roch Theorem \ref{lrr} allows us to rewrite this integrand as
\begin{align*}
\iota_{\theta_{i_{1}}}\ldots\iota_{\theta_{i_{k}}}\tau^{W}_{2r}(1,A,\ldots A)&={\rm ev}_{1}(\iota_{\theta_{i_{1}}}\ldots\iota_{\theta_{i_{k}}}\tau^{W}_{2r})(A,\ldots,A)\\
&=\iota_{\theta_{i_{1}}}\ldots\iota_{\theta_{i_{k}}}{\rm ev}_{1}(\tau^{W}_{2r})(A,\ldots,A))\\
&=\iota_{\theta_{i_{1}}}\ldots\iota_{\theta_{i_{k}}}\chi(P)(A,\ldots,A),
\end{align*}
where $P$ is the invariant polynomial \eqref{cherntodd} defining the Todd class. Next, in the evaluation of the Chern--Weil homomorphism, three types of curvature terms appear:
\[
C(A,A),\quad C(\theta_{i},A),\quad C(\theta_{i},\theta_{j}).
\]
As explained in \ref{LRR}, one has $C(a,b)=[\pi(a),\pi(b)]-\pi([a,b])$ where $\pi:\calW\rightarrow \mathfrak{gl}_n\oplus\mathfrak{gl}_r$ is the projection. 
The following lemma allows us to compute these terms.
\begin{lemma}
\label{lemma-fc}
Let $A(\nabla)\in \Omega^{1}_{\calL}(\calW)$ be the element defined in \eqref{fedcon}. Then:
\begin{itemize}
 \item[$i)$] $\deg_{\calL}(A(\nabla))= 1$ and $A(\nabla)$ has the form
\[
A(\nabla)=\sum_{k\geq -1} A_{k}\quad\quad\quad\deg_{\calL^{\vee}}(A_{k})=k+1,
\]
\item[$ii)$] The first few terms are given by:
\begin{align*}
A_{-1}&=-1\otimes {\rm id}_{\calL}\in  \widehat{\sym}(\calL^{\vee})\otimes \calL\otimes\calL^{\vee}\subset \Omega^{1}_{\calL}(\calW)\\
A_{0}&=0\hspace{2.5cm}\mbox{when $\nabla$ is torsion free}.
\end{align*}
\end{itemize}
\end{lemma}
\begin{proof}
We wrote $\deg_{\calL}$ and $\deg_{\calL^{\vee}}$ for the degree in $\sym(\calL)$ resp. $\widehat{\sym}(\calL^{\vee})$ of an element in the Weyl algebra $\calW(\calL)\cong \widehat{\sym}(\calL^\vee)\otimes\sym(\calL)$.
Since $d_{CE}$ and $\nabla$ act by derivations, $A$ also has this property. One forms that act by derivations can clearly be written
 in the above from, thus $ii)$ is proven. For $X\in\calL$ and $\alpha\in \sym^k(\calL^\vee)$ we have
\begin{equation}\label{Ai}
A_i(X)(\alpha)(X_1\ldots X_{k+i-1})=A(X)(\alpha)(X_1\ldots X_{k+i-1})
\end{equation}
and since $\sym(\calL^\vee)$ is generated by $\sym^1(\calL^\vee)$, we can assume $\alpha\in \sym^1(\calL^\vee)$.

 The definition of $j^{PBW}_\nabla$ allows to compute, for $\phi\in \calJ(\calL)$ and $\alpha\in \sym^1(\calL^\vee)$,
\begin{align*}
j_{\nabla}(\phi)(1)=\phi(1),\quad j_{\nabla}(\phi)(X)=\phi(X),\quad
j_{\nabla}(\phi)(XY)=\frac{1}{2}\phi(XY+YX-\nabla_XY-\nabla_YX).
\end{align*}
which in turn implies that
\begin{equation}\label{j-1}
j_\nabla^{-1}(\alpha)(X)=\alpha(X),\quad j_\nabla^{-1}(\alpha)(XY+YX-\nabla_XY-\nabla_YX)=\alpha(XY)=0.
\end{equation}
Using \ref{Ai} we can explicitly compute $A_{-1}$ and $A_0$:
\begin{align*}
A_{-1}(X)(\alpha)(1)=&j_\nabla\circ d_{CE}\circ j_{\nabla}^{-1}(\alpha)(1,X)\\
=&X( j_{\nabla}^{-1}(\alpha)(1))-j_{\nabla}^{-1}(\alpha)(X)\\
=&-\alpha(X).\\
A_0(X)(\alpha)(Y)=&j_\nabla\circ d_{CE}\circ j_{\nabla}^{-1}(\alpha)(Y,X)-\nabla_X\alpha(Y)\\
=&X( j_{\nabla}^{-1}(\alpha)(Y))- j_{\nabla}^{-1}(\alpha)(XY)-X(\alpha(Y))+\alpha(\nabla_XY)\\
=&X(\alpha(Y))- j_{\nabla}^{-1}(\alpha)(XY)-X(\alpha(Y))+\alpha(\nabla_XY).
\end{align*}
When the torsion $\nabla_XY-\nabla_YX-[X,Y]$ vanishes equation \ref{j-1} implies $j_\nabla^{-1}(\alpha)(XY)=\alpha(\nabla_XY)$, and 
hence the result.
\end{proof}

\begin{remark}
Because of the previous Lemma we have $\pi(A)=0$ in the computation of the Chern--Weil homomorphism. In the twisted case,  c.f. remarks \ref{twisted-remark} and \ref{twisted-remark2} things are a little different. First note that  in this case $\pi(A)\in \Omega_\calL(\End(\calL))$, in other words, the projection on the $\mathfrak{gl}_r$ part, c.f. 
\ref{LRR}, vanishes because $A$ acts by derivations and thus satisfies $\deg_\calL=1$. Then we define $\tilde{A}=A-\pi(A)$ which clearly satisfies $\pi(\tilde{A})=0$.
Observe that, using $\tilde{\nabla}^{\calL}=\nabla^{\calL}-\pi(A)$, we can write this as follows:
\[
\tilde{A}_{\underline{t}}=j_{\nabla^{\calL}_{\underline{t}},\nabla^{\calE}_{\underline{t}}}\circ d_{CE}\circ j_{\nabla^{\calL}_{\underline{t}},\nabla^{\calE}_{\underline{t}}}-\tilde{\nabla}^{\calL}_{\underline{t}}-\nabla^{\calE}_{\underline{t}}.
\]
Since $\pi(A)$ takes values in $\End(\calL)$ and $\tau_{2r}^w$ is basic with respect to these elements, the following equality holds:
\[
\iota_{\theta_{i_{1}}}\ldots\iota_{\theta_{i_{k}}}\tau^{W}_{2r}(1,A,\ldots ,A)=\iota_{\theta_{i_{1}}}\ldots\iota_{\theta_{i_{k}}}\tau^{W}_{2r}(1,\tilde{A},\ldots ,\tilde{A}),
\]
hence we can assume that $\pi(A)=0$.
\end{remark}

We proceed with the proof of theorem \ref{index-theorem}. Either using torsion free connections or the previous remark, 
we can assume that $\pi(A)=0$.  Moreover, from \ref{depconn} we obtain an explicit 
expression for $\theta_i$, which first terms are given by:
\[
\theta_i=\gamma_i+\frac{1}{2}\nabla^{\rm  aff}\gamma_i+\frac{1}{6}(\nabla^{\rm  aff})^2\gamma_i+\ldots
\]
Here $\gamma_i$ is a derivation of $\widehat{\sym}(\calL^\vee)$ that raises the degree by one. Therefore $\deg_\calL(\gamma)=1$ and $\deg_{\calL^\vee}(\gamma)=2$. The other terms all have higher degrees, hence $\pi(\theta_i)=0$. This implies that $C(A,A)=-\pi([A,A])$, $C(A,\theta_i)=-\pi([A,\theta_i])$ and
$C(\theta_i,\theta_j)=\pi([\theta_i,\theta_j])$. Since the commutator of derivations of degree $(p,1)$ and $(q,1)$ has degree $(p+q,1)$, it follows that
$\pi([\theta_i,\theta_j])=0$. Using the expressions from \ref{depconn} we obtain:
\begin{align*}
C(A,A)&=-\pi([A,A]=\pi(\nabla A+R(\nabla^{\rm  aff}))=\pi(R(\nabla^{\rm  aff}))\\
C(\theta_i,A)&=-\pi([\theta_i,A])=-\pi(\partial_{t_i}A+\nabla^{\rm  aff}\theta_i+\partial_{t_i}\nabla^{\rm  aff})=-\gamma_i
\end{align*}
We used here that both the operators $\nabla^{\rm  aff}$ and $\partial_{t_j}$ leave the $\calL$ and $\calL^\vee$ degrees intact. 
Now we identified all the curvature terms, we can compare them with the terms that appear in the curvature of $\nabla_k^{\rm tot}=d_k+\sum_{i=0}^kt_i\nabla_i$.
An easy computation shows that 
\[
R(\nabla^{\rm tot}_k)=R(\nabla^{\rm  aff})+\sum_{i=1}^kdt_i\gamma_i
\]
Thus, the class obtained from Chern-Weil theory is computed  by
\begin{align*}
\int_{\Delta^k}P(R(\nabla^{\rm tot}_k))(X_1,\ldots, X_{2n-k})=\int_{\Delta^k}P(R(\nabla^{\rm tot}_k))(X_1,\ldots X_{2n-k},\frac{\partial}{\partial t_1},\ldots, \frac{\partial}{\partial t_k})
dt_1\ldots dt_k \\
=\int_{\Delta^k}\frac{1}{(2n)!}\sum_{\sigma\in S^{2n}}(-1)^{\sigma}\tilde{P}\left(R(\nabla^{\rm tot}_k)(Y_{\sigma(1)},Y_{\sigma(2)}),\ldots,R(\nabla^{\rm tot}_n(Y_{\sigma(2k-1)},Y_{\sigma(2n)}))\right)dt_1\ldots dt_k
\end{align*}
Given the form of $R(\nabla^{\rm tot})$, it is easy to see that only permutations where $Y_{\sigma(2l-1)},Y_{\sigma(2l)}=\partial_{t_i},X_j$ or
$X_j,X_k$ have nonzero contributions, namely $\sum_i\gamma_i(X_j)$ and $R(\nabla^{\rm  aff}(X_j,X_k))$ . On the other hand the class
$\Phi_{\calL}(1)$ is given by
\begin{align*}
&\int_{\Delta^k}\chi(P)(\theta_{i_1}\wedge\ldots\wedge\theta_{i_k}\wedge A(X_1)\wedge\ldots\wedge A(X_{2n-k})dt_1\ldots dt_k\\
=&\int_{\Delta^k}P\left(\frac{1}{(2n)!}\sum_{\sigma\in S^{2n}}(-1)^{\sigma}(C(B_{\sigma(1)},B_{\sigma(2)}),\ldots C(B_{\sigma(2n-1)},B_{\sigma(2n)}))\right)dt_1\ldots dt_k
\end{align*}
where the only nonzero contributions come from permutations with $B_{\sigma(2l-1)},B_{\sigma(2l)}=\theta_i, A(X_j)$ or
$A(X_{i}),A(X_{j})$, and in those cases the curvature is given by $\gamma_i(X_{j})$ and $R(\nabla^{\rm  aff})(X_{i},X_{j})$ and $\gamma_i$.
\end{proof}

\subsection{Compatibility with the HKR map}
In this section we shall refine the index Theorem \ref{index-theorem} of the previous section by showing that the Lie algebra cohomology class it determines measures
the obstruction for our character map to be compatible with the Hochschild--Kostant--Rosenberg map \eqref{hkr}. Define $HKR_{\calL}:=\rho^{*}\circ HKR$, the pre-composition
of the HKR-map with the pull-back along the anchor, i.e., 
\[
HKR_{\calL}(f_{0}\otimes\ldots\otimes f_{k}):=f_{0}d_{\calL}f_{0}\wedge\ldots\wedge d_{\calL}f_{k}\in\Omega^{k}_{\calL},\quad f_{i}\in\calO_{X},~i=0,\ldots k.
\]
Using the obvious inclusion $\calO_{X}\hookrightarrow \calU(\calL)$, we can compare this map with the character map of Theorem \ref{character-map}:
\begin{theorem}
\label{comp-hkr}
Let $\calL$ be a Lie algebroid over a locally ringed space $(X,\calO_X)$. 
The diagram
\[
\xymatrix{CC_\bullet(\calO_X)\ar[r]^{HKR_{\calL}}\ar[d]_{i^*}&\Omega_{\calL}^\bullet\ar[d]^{\wedge{\rm Td}_\calL(\calL){\rm Ch}_\calL(\calE)}
\\
CC_\bullet(\calU(\calL;\calE))\ar[r]_{\Phi_\nabla}&\Omega^\bullet_\calL}
\]
commutes in the derived category.
\end{theorem}

\begin{proof}

First remark that in the definition of the character map $f\in\calO_{X}$ is embedded in the sheaf of Weyl algebras $\calW$ as
\[
j_{\nabla}(f)=e^{\nabla^{s}}f=j_{\nabla}^{0}(f)+j_{\nabla}^{1}(f)+\ldots\in\widehat{\sym}(\calL^{\vee}).
\]
Analogous to the proof of Theorem \ref{comp-hkr}, we now have to consider the forms
\begin{equation}
\label{hkr-comp}
\tag{$\star$}
\iota_{\theta_{1}}\ldots\iota_{\theta_{p}}\tau^{W}_{2r}((A)_{2r-k-p}\times j_{\nabla}(f_{0})\otimes \ldots\otimes j_{\nabla}(f_{k})).
\end{equation}
We will evaluate these forms using the expansion of $A(\nabla)$ from Lemma \ref{lemma-fc} and the expansion of $\theta_i$ given in 
\ref{depconn}.
Write the cocycle \eqref{Hoch-cocycle} as $\tau_{2r}=\mu_{2r}\circ S_{2r}\circ \pi_{2r}$.  Fix a local basis $\{e_{i},~i=1,\ldots, r\}$ of $\calL$,
and denote the dual basis of $\calL^{\vee}$ by $\{de_{i}\}$. Recall the notation  $\deg_{\calL^\vee}(a)$ and $\deg_\calL(a)$ for $a\in \calW(\calL)$.
For elements $a_0\otimes\ldots\otimes a_k$ we have $\deg_\calL(a_0\otimes \ldots \otimes a_k)=\sum_i \deg_\calL(a_i)$ and $\deg_{\calL^\vee}(a_0\otimes \ldots \otimes a_k)=\sum_i \deg_{\calL^\vee}(a_i)$. 
Note that $\mu_{2r}$ is the projection on the $\deg_\calL=\deg_{\calL^\vee}=0$ part of $a_0\otimes\ldots a_{2r}$ followed by multilplication.
We claim that in the expression of $\star$ we can make, without changing the resulting espression, the following replacements;
\begin{align*}
A\rightsquigarrow& A_{-1} + A_1\\
j_{\nabla}(f_0)\rightsquigarrow& j_{\nabla}^{0}(f_{0})\\
 j_{\nabla}^{1}(f_{i})\rightsquigarrow& j_{\nabla}^{1}(f_{i}),\quad i\geq 1 \\
\theta_l\rightsquigarrow& \theta_{l}^{2}.
\end{align*}
By $\theta_l^2$ we mean the $\deg_{\calL^\vee}=2$ part of $\theta_l$. Consider $A_k$ with $k\geq 2$ in the $i$'th slot. We have $\deg_\calL(A_k)=1$ and $\deg_{\calL^\vee}(A_k)=k+1$, 
hence an application of $\pi_{2r}$ will give an element of either $(\deg_{\calL^\vee},\deg_\calL)=(k+1,0)\;\;\text{or}\;\; (k,1)$. An application 
of $e^{s\alpha_{ij}}$ will only give a nonzero $(0,0)$ term if the term in the $j$'th slot has degree $(0,k)$ or $(1,k-1)$, and 
these terms are not present. For $\theta_l$ an analogous reasoning applies. Now consider $j_{\nabla}(f_i)$ in the $i$'th slot. The term $j^0_\nabla(f_i)$ does not
survive $\pi_{2r}$, and $\pi_{2r}$ applied to terms $j^k_{\nabla}$ with $k\geq 2$ gives terms of degree $(k-1,0)$, to which an application of $e^{s\alpha_{ij}}$
only gives a nonzero $(0,0)$ term if the term in the $j$'th slot has degree $(0,k-1)$. Again, these terms are not present; here it is used that $A_0=0$.
In conclusion:
\begin{align*}
(\star)\;=\;\iota_{\theta_1^2}\ldots\iota_{\theta_p^2}\tau^{W}_{2r}((A_{-1}+A_{1})_{2r-k-p}\times j_{\nabla}^0(f_0)\otimes j_{\nabla}^1(f_1)\otimes\ldots\otimes j_{\nabla}^1(f_k)).
\end{align*}
We can still eliminate terms from this expression. We have for general elements $a_0\otimes\ldots\otimes a_{2r}\in \calW(\calL)\otimes\ldots\otimes\calW(\calL)$:
\begin{align*}
\deg_{\calL^\vee}(\alpha_{ij}(a_0\otimes\ldots\otimes a_{2r}))&=\deg_{\calL^\vee}(a_0\otimes\ldots\otimes a_{2r})-1\\
\deg_{\calL^\vee}(\pi_{2r}(a_0\otimes\ldots\otimes a_{2r}))&=\deg_{\calL^\vee}(a_0\otimes\ldots\otimes a_{2r})-r
\end{align*}
and similar formulas for the $\calL$ degree. This implies that the only terms in $a_0\otimes\ldots\otimes a_{2r}$ that have
a nonzero contribution after an application of $\tau_{2r}$ have equal $\calL$ and $\calL^\vee$ degrees. 
It follows that
\begin{align*}
(\star)\;=\tau^{W}_{2r}((\theta^2_p)\times\ldots\times(\theta^{2}_{1})\times (A_{1})_{r-k-p}\times (A_{-1})_{r}\times j_{\nabla}^0(f_0)\otimes j_{\nabla}^1(f_1)\otimes\ldots\otimes j_{\nabla}^1(f_k)).
\end{align*}
The first order jets $j_{\nabla}^{1}(f_{i})=\sum_{j}\rho(e_{j})(f_{i})de_{j},~i=1,\ldots,k$ of the $f_{i}$ combine
with $A_{-1}=\sum_{i}1\otimes e_{i}\otimes de_{i}$, after applying $\pi_{2r}$, to the forms $\rho^{*}(df_{i})$, so that we have:
\begin{align*}
(\star)\;=&\;\frac{1}{k!}f_{0}\rho^{*}df_{1}\wedge\ldots\wedge \rho^{*}df_{k}\pi_{2r} ((\theta_{p})\times\ldots\times (\theta_{1})\times\underbrace{\pi\times\ldots\times\pi}_{\#=k}\times (A_{-1})_{r-k}\times (A_{1})_{r-k-p})\\
=&f_{0}\rho^{*}df_{1}\wedge\ldots\wedge \rho^{*}df_{k}\iota_{\theta_{1}}\ldots\iota_{\theta_{p}}\frac{\iota^{k}_{\pi}}{k!}\tau^{W}_{2r}((A_{-1})_{r-k}\times (A_{1})_{r-k-p})\\[10 pt]
=&f_{0}\rho^{*}df_{1}\wedge\ldots\wedge \rho^{*}df_{k}u^{-n}[\Td_{\calL}(\calL)\Ch(\calE)]
\end{align*}
Where we used the definition of the cyclic cocycle using the insertion $\iota_{\pi}$, and
Theorem \ref{index-theorem}.
\end{proof}

\subsection{Holomorphic Lie algebroids}
In this final subsection we give an alternative derivation of our main index theorem for the special case of a holomorphic Lie algebroid. 
This alternative approach is possible because in this case there is, besides the \v{C}ech resolution used so far, an alternative Dolbeault-type resolution
to compute the Lie algebroid cohomology constructed in \cite{lgsx}. 
Let $X$ be a complex manifold with holomorphic tangent bundle $TX$ and structure sheaf $\calO_{X}$ consisting of germs of holomorphic functions. 
A {\rm holomorphic Lie algebroid} is a holomorphic vector bundle
$L\to X$ equipped with a vector bundle map $\rho:A\to TX$ so that the sheaf of holomorphic sections $\calL$ becomes a Lie algebroid 
in the sense of section \ref{def-la}. 

We start by recalling the Dolbeault complex of \cite{lgsx}: introduce
\[
\calA^{p,q}_{L}:=\Gamma^{\infty}\left(X,\bigwedge^{p} (T^{0,1}X)^{\vee}\otimes\bigwedge^{q}(L^{1,0})^{\vee}\right).
\]
The structure of the pair $(T^{0,1}X,L^{1,0})$, called a ``matched pair of Lie algebroids'' in \cite{lgsx}, then gives:
\begin{itemize}
\item[$i)$] As a smooth vector bundle, $L^{1,0}$ is a module over $T^{0,1}X$. The same is true for all duals and exterior powers of $L^{1,0}$, and results in a differential
 $\bar{\partial}_{L}:\calA_{L}^{p,q}\to \calA_{L}^{p+1,q}$ given by the Lie algebroid differential \eqref{Koszul} of the complex
Lie algebroid $T^{0,1}X$ with coefficients in the module $\bigwedge^{q}(L^{1,0})^{\vee}$.
\item[$ii)$] The anchor map $\rho$ defines a $L^{1,0}$-module structure on $T^{0,1}X$ by the formula $\nabla_{\alpha}X:=pr_{0,1}([\rho(\alpha),X])$.
 Again, this carries over to exterior powers of the dual, and gives the 
second differential $d_{L}:\calA^{p,q}\to\calA_{L}^{p,q+1}$ with coefficients in 
$\bigwedge^{p}(T^{0,1}X)^{\vee}$.
\end{itemize}
The two differentials commute, $[\bar{\partial}_{L},d_{L}]=0$, so we can form the total complex.
The kernel of the differential
$\bar{\partial}_{L}:\calA_{L}^{0,\bullet}\to \calA_{L}^{1,\bullet}$ is precisely the holomorphic Lie algebroid complex $(\Omega_{L}^{\bullet},d_{L})$, and the main 
theorem of \cite{lgsx} states that $\calA_{L}^{\bullet,\bullet}$ is a fine resolution of this complex so that the total complex of $\calA_{L}^{\bullet,\bullet}$ computes
the Lie algebroid cohomology of $(\calL,\calO_{X})$. In the case $L^{1,0}=T^{1,0}X$ of the holomorphic tangent bundle, this is just the Dolbeault resolution.

To apply this resolution for our construction of the character map, we tensor with the smooth jet bundle and universal enveloping algebra to obtain the complex
\begin{equation}
\label{Dolbeault-dga}
\left(\calJ^{\infty}_{L^{1,0}}\otimes_{C^{\infty}_{X_{2}}}\calU(L^{1,0})\otimes_{C^{\infty}_{X_{1}}}\calA^{p,q}_{L},\nabla^{(1)}+\bar{\partial}_{L}\right).
\end{equation}
This time the kernel of the differential $\bar{\partial}_{L}$ in degree $0$ is precisely the sheaf of DGA's of Lemma \ref{lemdga} for the Lie algebroid $\calL$. 

Next we choose a torsion free $L^{1,0}$-connection $\nabla$ on $L^{1,0}$  Remark that in this $C^{\infty}$-setting, such a connection always. exists. By our PBW-theorem,
this induces an isomorphism
\[
j_{\nabla}:\calJ^{\infty}_{L^{1,0}}\stackrel{\cong}{\longrightarrow}\widehat{\sym}(L^{1,0})^{\vee}.
\]
Using \ref{ncpbw}, this induces an isomorphism of the total space of the  DGA \eqref{Dolbeault-dga} with $\calW_{L^{1,0}}\otimes_{C^{\infty}_{X}}\calA_{L}^{\bullet,\bullet}$, also denoted
$j_{\nabla}$, and allows us to introduce
\[
A(\nabla):=j_{\nabla}\circ \nabla^{(1)}\circ j_{\nabla}^{-1}-(\nabla+\overline{\partial}_{L})
\]
Remark that the flat connection $\nabla^{(1)}+\overline{\partial}_{L}$ on $\calJ^{\infty}_{L^{1,0}}\otimes_{C^{\infty}_{X_{2}}}\calU(L^{1,0})$, 
when transferred to $\calW_{L^{1,0}}$, does not necessarily have the form $j_{\nabla}\circ \nabla^{(1)}\circ j_{\nabla}^{-1}+\overline{\partial}_{L}$ and
hence $A$ has both components in the $T^{0,1}$ and $L^{1,0}$ direction. This would be the case if a global holomorphic connection would exist, because 
then the $T^{0,1}$-module structures on $\calJ^{\infty}_{L^{1,0}}$ and $\calW_{L^{1,0}}$ are be respected by $j_{\nabla}$.

\begin{theorem}
Let $L\to X$ be a holomorphic Lie algebroid over a complex manifold, and $E\to X$ a holomorphic vector bundle. 
\begin{itemize}
\item[$i)$] A pair $(\nabla^{L},\nabla^{E}$ of $L$-connections on $L$ and $E$ compatible with the holomorphic structure, with $\nabla^{L}$ torsion free,  induces a character map
\[
\Phi_{\nabla}:\left(CC_{\bullet}(\calU(\calL;E)),b+uB\right)\to \left(\bigoplus_{p+q=2r-\bullet}\calA^{p,q}_{L}[u],d_{L}+\bar{\partial}_{L}\right)
\]
\item[$ii)$] For two connections $\nabla$ and $\nabla'$, the maps $\Phi_{\nabla}$ and $\Phi_{\nabla'}$ are chain homotopic.
\item[$iii)$] {\rm (Index Theorem)} The following equality of differential forms holds true:
\[
\Phi_{\nabla}(1)={\rm Td}_{L}{\rm Ch}_{E}
\] 
\end{itemize}
\end{theorem}

\appendix
\section{Hochschild and cyclic homology}
\label{cyclic}
\subsection{Hochschild theory}
Here we briefly recall the basic notions of Hochschild and cyclic homology as well as their dual cohomology versions. An extensive introduction into the subject is given in \cite{loday}. Let $A$ be a unital algebra over a field $\mathbb{K}$, and 
write $\bar{A}:=A\slash 1\C$. For any $A$-bimodule $M$, the (reduced) Hochschild chain complex is given by the graded vector space $C_\bullet(A;M)$, where
$C_k(A):=M\otimes \bar{A}^{\otimes}$, with differential $b:C_k(A;M)\to C_{k-1}(A;M)$ given by
\begin{align*}
b(m\otimes &a_{1}\otimes \ldots\otimes a_k):=m\cdot a_{1}\otimes a_{2}\otimes\ldots\otimes a_{k}+\\
&+\sum_{i=1}^{k-1} (-1)^im\otimes a_{1}\otimes\ldots\otimes a_ia_{i+1}\otimes\ldots\otimes a_k+(-1)^ka_k\cdot m\otimes a_1\otimes\ldots\otimes a_{k-1}.
\end{align*}
One easily checks that $b^2=0$ and the resulting homology groups are called the Hochschild homology
of $A$ with values in $M$, written $H_\bullet(A;M)$.

Dually, the cochain complex with $C^{k}(A;M):=\Hom_{\mathbb{K}}(A^{\otimes k},M)$ and the dual differential (also written $b$), define the Hochschild cohomology groups 
$H^{\bullet}(A;M)$.

\subsection{Cyclic theory}
There exists a degree increasing differential on the Hochschild chain complex $B:\overline{C}_k(A)\to \overline{C}_{k+1}(A)$
given by
\[
B(a_0\otimes\ldots\otimes a_k):=\sum_{i=0}^k(-1)^{ik}1\otimes a_i\otimes a\ldots\otimes a_k\otimes a_0\otimes\ldots \otimes a_{i-1}.
\]
This operator satisfies $B^2=0$ and $[b,B]=0$. To define the chain complexes for the various versions of cyclic homology, 
we add, following \cite{gj} a formal variable $u$ of degree $-2$ and set
\begin{align*}
CC_{\bullet}^{W}(A):=C_{\bullet}(A)[[u]]\otimes_{\gf[u]}W,
\end{align*}
where $W$ is a $\gf[u]$-module, and equip it with the differential $b+uB$. The various choices for $W$ lead to the following theories:
\begin{itemize}
\item $W=\gf$ gives us back the Hochschild homology.
\item $W=\gf[u,u^{-1}]\slash u\gf[u]$ leads to (ordinary) cyclic homology, denoted by $HC_{\bullet}(A)$.
\item $W=\gf[u,u^{-1}]$ leads to periodic cyclic homology, denoted $HP_{\bullet}(A)$. It has period $2$ with periodicity operator induced by multiplication by $u$.
\item $W=\gf[u]$ leads to negative cyclic homology, denoted $HC^{-}_{\bullet}(A)$.
\end{itemize}
Short exact sequences of $\gf[u]$-modules lead to long exact sequences in homology. For example, the usual SBI-sequence is induced from
the short exact sequence 
\[
0\longrightarrow \gf[u,u^{-1}]\slash \gf[u]\stackrel{S}{\longrightarrow} \gf[u,u^{-1}]\slash u\gf[u]\longrightarrow \gf\longrightarrow 0. 
\]
On the other hand, the short exact sequence
\[
0\longrightarrow u\gf[u]\longrightarrow \gf[u,u^{-1}]\longrightarrow \gf[u,u^{-1}]\slash u\gf[u]\longrightarrow 0
\]
leads to the long exact sequence 
\[
\ldots \longrightarrow HC^{-}_{n+2}(A)\longrightarrow HP_{n}(A)\longrightarrow HC_{n}(A)\longrightarrow HC_{n+1}^{-}\longrightarrow \ldots
\]
relating negative and periodic cyclic homology.
Finally, for the cohomology theories we use the dual complexes $CC^{\bullet}_{W}(A):=C_{\bullet}(A)[[u^{-1}]]\otimes_{\gf[u^{-1}]}W$ equipped with the differential $b+u^{-1}B$.

\subsection{Operations on Hochschild and cyclic cohomology}
In this paper we need a few operations on the Hochschild and cyclic (co)chain complexes, which are part of a more richer ``calculus''.
Let $A$ be a unital algebra and pick $a\in A$.
The insertion operator $\iota_{a}:C_k(A)\rightarrow C_{k+1}(A)$ is defined by
\[
\iota_a a_0\otimes\dots\otimes a_k:=\sum_{i=0}^{k}(-1)^{i+1} a_0\otimes\dots\otimes a_i\otimes a\otimes a_{i+1}\otimes\dots\otimes a_k)
\]
Then there is a lie algebra action of $\mathfrak{gl}(A)$, the algebra $A$ equipped with the commutator bracket,  on $C_k(A)$:
\begin{equation}
L_a=[b,i_a],
\end{equation}
which is explicitly given by
\[
L_a(a_0\otimes\dots\otimes a_k)=\sum_{i=0}^{k}(a_0\otimes\dots\otimes[a,a_i]\otimes \dots,a_k).
\]
Dually, these formulas induce actions on $C^k(A)$, and extending them by linearity induces actions on the cyclic and periodic complexes.\\
 The shuffle product, only defined when $A$ is a unital, graded algebra, of $a_0\otimes\dots\otimes a_p$ and $b_0\otimes\dots\otimes b_q$ is given by
\begin{align*}
(a_0\otimes\dots\otimes a_p)\times (b_0\otimes\dots\otimes b_q)=&\\
=(-1)^{deg(b_0)(\sum_i deg(a_i))}Sh_{p,q}&(a_0b_0\otimes a_1\otimes\dots\otimes a_p\otimes b_1\otimes\dots\otimes b_q),
\end{align*}
where
\[
Sh_{p,q}(c_0\otimes\dots\otimes c_{p+q})=\sum_{\sigma\in S_{p,q}}sgn(\sigma)c_0\otimes c_{\sigma(1)}\otimes\dots c_{\sigma(p+q)}
\]
and $S_{p,q}\subset S_{p+q}$ consists of the $p,q$ shuffle permutations in $S_{p+q}$.

\subsection{The map to Lie algebra cohomology}
\label{mor-lie}
Let $\g$ be a Lie algebra over $\gf$. The Lie algebra cochain complex, written  $C^{\bullet}_{\rm Lie}(\g)$, is a special case of the Lie--Rinehart complex \eqref{Koszul} with base ring $R=\gf$, $L=\g$ and $M=\gf$. 
Consider the cyclic cochain complex $CC^{\bullet}_{W}(A)$ for a $\gf[u^{-1}]$-module $W$. The evaluation map 
\[
{\rm ev}_{1}(\phi)(a_{1},\ldots,a_{k}):=\sum_{\sigma\in S_{k}}(-1)^{\sigma}\phi(1\otimes a_{\sigma(1)}\otimes\ldots\otimes a_{\sigma(k)}),\quad a_{1},\ldots a_{k}\in A,
\]
defines a morphism of cochain complexes
\begin{equation}
\label{mtolie}
{\rm ev}_{1: }\left(CC^{\bullet}_{W}(A),b+u^{-1}B\right)\longrightarrow \left(C^{\bullet}_{\rm Lie}(\mathfrak{gl}(A))[[u^{-1}]]\otimes_{\gf[u^{-1}]}W,d_{\rm Lie}\right),
\end{equation}
where $\mathfrak{gl}(A)$ denotes the Lie algebra associated to $A$, i.e., $A$ equipped with the commutator.

\section{The local Riemann--Roch theorem}
\label{LRR}
\subsection{The Weyl algebra}
\label{weyl}
Let $\calO_{n}=\gf[x^{1},\ldots,x^{n}]$ be the polynomial algebra on $n$ generators. The Weyl algebra $W_{n}$ is,
by definition, the algebra of differential operators on $\calO_{n}$. Over $\calO_{n}$, it is generated by the fundamental
derivations $\partial\slash\partial x^{i},~i=1,\ldots,n$ subject to the well-known commutation relations
\[
[\frac{\partial}{\partial x^{i}},x^{j}]=\delta_{ij}.
\]
It therefore admits another presentation as the space of polynomials $\gf[q^{1},\ldots,q^{n},p_{1},\ldots,p_{n}]$ equipped with the Moyal--Weyl product
\begin{equation}
\label{moyal-weyl}
f\star g:=m\circ e^{\pi}(f\otimes g),\quad f,g\in \gf[q^{1},\ldots,q^{n},p_{1},\ldots,p_{n}],
\end{equation}
where $m$ is the commutative product of polynomials and
\[
\pi:=\sum_{i=1}^{n}\left(\frac{\partial}{\partial p_{i}}\otimes\frac{\partial}{\partial q^{i}}-\frac{\partial}{\partial q^{i}}\otimes \frac{\partial}{\partial p_{i}}\right).
\]
The isomorphism with $W_{n}$ is simply given by sending $x^{i}\mapsto q^{i},~p_{i}\mapsto \partial\slash\partial x^{i}$. 
The center of the Weyl algebra is clearly equal to $\gf$, and, since every derivation is inner, we have the exact sequence of Lie algebras
\begin{equation}
\label{derw}
0\longrightarrow\gf\longrightarrow W_{n}\stackrel{\ad}{\longrightarrow}\Der(W_{n})\longrightarrow 0.
\end{equation}

\subsection{The Hochschild cocycle} 
Using the spectral sequence associated to the filtration given by the degree of a polynomial, it was proved in \cite{ft} that
\[
HH_{\bullet}(W_{n})=\begin{cases} \gf & \bullet=2n\\ 0&\bullet\not = 2n.
\end{cases}
\]
A generator for the nontrivial class in degree $2n$ is given by
\[
c_{2n}:=\sum_{\sigma\in S_{2n}}(-1)^{\sigma}1\otimes y_{\sigma(1)}\otimes\ldots y_{\sigma(2n)},\quad y_{2i}=q_{i},~y_{2i-1}=p_{i},~i=1,\ldots,n.
\]
In \cite{ffs}, a dual Hochschild cocycle generating the only nontrivial class was defined by the formula
\begin{equation}
\label{Hoch-cocycle}
\tau_{2n}^{\rm Hoch}(a):=\mu_{2n}\int_{\Delta^{2n}}\prod_{0\leq i<j\leq 2n}e^{(t_{j}-t_{i}-\frac{1}{2})\pi_{ji}}(1\otimes\pi^{\wedge n})(a)dt_{1}\cdots dt_{2n},
\end{equation}
where $a=a_{0}\otimes\ldots a_{2n}\in W_{n}^{\otimes(2n+1)}$ and $\mu_{2n}(a_{0}\otimes\ldots a_{2n})=a_{0}(0)\cdots a_{2n}(0)$ is the evaluation
at zero followed by the commutative multiplication. That this cocycle exists follows simply by duality, and indeed we have the property that
\begin{equation}
\label{normalization}
\tau_{2n}(c_{2n})=1.
\end{equation}
However, the explicit expression above 
allows to prove the surprising fact the cocycle $\tau_{2n}^{\rm Hoch}$ is $\mathfrak{gl}(n,\gf)$-invariant and basic: 
\begin{equation}
\label{basic}
L_X\tau_{2n}^{\rm Hoch}=0,\quad \tau_{2n}^{\rm Hoch}(\ldots,X,\ldots)=0,\quad X\in \mathfrak{gl}(n,\gf).
\end{equation}

\subsection{The fundamental cyclic cocycle}
In \cite{ppt} and \cite{willwacher} the Hochschild cocycle $\tau_{2n}^{\rm Hoch}$ was extended to a full cyclic cocycle in the $(b,B)$-complex.
Here we follow the presentation of \cite{willwacher}. Define the insertion operator 
\[
\iota_{\pi}=\sum_{i=1}^{n}\iota_{p_{i}}\iota_{q^{i}}:C^{k}(W_{n},W_{n}^{*})\to C^{k-2}(W_{n},W_{n}^{*}).
\]
With this notation, for each choice of $w\in W\backslash\{0\}$, the cochain 
\[
\tau_{2n}^{w}:=e^{-u^{-1}\iota_{\pi}}\tau^{\rm Hoch}_{2n}\otimes w\in CC^{2n}_{W}(W_{n})
\]
is closed: $(b+u^{-1}B)\tau_{2n}^{w}=0$. For $w=1$ in ordinary cyclic homology one can expand
\[
\tau_{2n}^{1}=\tau^{\rm Hoch}_{2n}+u^{-1}\tau_{2n-2}+\ldots + u^{-n}\tau_0\in C^{2n}(D_n)[[u,u^{-1}]\slash u C^{2n}(D_n)[[u]],
\]
with $\tau_{2n-2k}=\iota_{\pi}^{k}\tau_{2n}^{\rm Hoch}\slash k!$. It is possible to perform the integration 
to obtain an explicit formula for for the lower degrees $\tau_{2k}$ as an integral over the simplex $\Delta^{2k}$ analogous to \eqref{Hoch-cocycle}, c.f.\ \cite{ppt}.

For the twisted case, we need a slight generalization of this cocyle. 
The algebra $W_{n}^{r}:=M_r(W_{n})$ is Morita equivalent to $W_{n}$, thus has the same Hochschild and cyclic homology. 
The cyclic cocyle can be generalized by the formula
\[
\tau^{w,r}((a_0\otimes M_0)\otimes\ldots\otimes (a_k\otimes M_k)):=\tau^{w}_{2n}(a_0\otimes\ldots \otimes a_k){\rm tr}(M_0\ldots M_k).
\]
The algebra $W_{n}^{r}$ contains the Lie subalgebra $\mathfrak{sp}_{2n}\oplus \mathfrak{gl}_{r}$, and the cocycle is basic and 
invariant w.r.t. this subalgebra, c.f. \cite{willwacher}.

\subsection{The local Riemann--Roch theorem in Lie algebra cohomology}
The most fundamental property of the cyclic cocycle $\tau_{2n}^{w}$ is the local Riemann--Roch theorem which links it in Lie algebra cohomology 
the characteristic classes appearing in the index theorem. This theorem has a long history dating back to \cite[Thm 5.1.1.]{ft} where it first appeared in abstract form for 
Hochschild class of $\tau_{2n}^{\rm Hoch}$. Other appearances of the theorem are in \cite{bnt,ffs,ppt,willwacher}. We shall state the theorem below as an equality 
on the level of {\em chains}, the proof follows as in \cite{willwacher} by evaluating the integrals appearing in the formula for $\tau_{2n}^{w}$. 

Let us first recall the Chern--Weil homomorphism in Lie algebra cohomology: Let $\h\subset\g$ be an inclusion of Lie algebras, and suppose that
there exists an $\h$-equivariant projection $\pi:\g\to\h$. Then $\pi$ determines, in the language of Lie algebroids, a $\g$-connection on $\h$ by 
the formula 
\[
\nabla_{X}Z=\pi([X,Z]), \quad \mbox{for}~ X\in\g, Z\in\h,
\]
with curvature $R\in\Hom(\bigwedge^{2}\g,\h)$ given by
\[
C(X,Y)=[\pi(X),\pi(Y)]-\pi([X,Y]),\quad X,Y\in\g.
\]
The Chern--Weil construction, explained in \S \ref{cw} for the more general case of Lie--Rinehart algebras then leads to a homomorphism 
\[
\chi:\left(\sym^{\bullet}\h^{\vee}\right)^{\h}\to C^{2\bullet}_{\rm CE}(\g,\h),
\]
explicitly given by the formula
\[
\chi(P)(X_{1},\ldots,X_{2k}):=\frac{1}{k!}\sum_{\sigma\in S_{2k}\atop \sigma(2i-1)<\sigma(2i)}(-1)^{\sigma}P\left(C(X_{\sigma(1)},X_{\sigma(2)}),\ldots,C(X_{\sigma(2k-1)},X_{\sigma(2k)})\right),
\]
where $X_{1},\ldots,X_{2k}\in\g$ and $P\in \sym^{k}\h^{\vee}$.  For the Riemann--Roch theorem we need this construction for the inclusion 
\[
\mathfrak{gl}_{n}\oplus\mathfrak{gl}_{r}\subset \mathfrak{gl}(D_{n})\oplus\mathfrak{gl}_{r}(\calO_{n}).
\]
Recall the invariant polynomials \eqref{cherntodd} defining the Chern and Todd classes.
\begin{theorem}[Local Riemann--Roch] 
\label{lrr}
Under the $u$-linear extension of the morphism \eqref{mtolie}, we have the following equality of Lie algebra cochains:
\[
ev_{1}(\tau_{2k}^{w,r})=(-1)^{k}\chi(\Td_{n}\Ch_{r})_{2k}\otimes w\in C^{2k}\left(\mathfrak{gl}(W_{n})\oplus \mathfrak{gl}_{r}(\calO_{n});\mathfrak{gl}_{n}\oplus \mathfrak{gl}_{r}\right).
\]
\end{theorem}

\end{document}